\documentclass{amsart}

\usepackage[english]{babel}
\usepackage[utf8x]{inputenc}
\usepackage{array}
\usepackage{graphicx}
\usepackage{enumerate}
\usepackage{color}
\usepackage[small]{caption}
\usepackage[normalem]{ulem}
\usepackage{latexsym,amssymb,amsmath,amsopn,amsthm,amscd,amsfonts,url}
\usepackage{tikz,pgflibraryshapes}
\usetikzlibrary{calc}
\usetikzlibrary{intersections}
\usetikzlibrary{arrows,shapes,positioning}
\usetikzlibrary{decorations.markings}

\usepackage[pdftex,pdfcreator={LaTeX},pdfproducer={PDFLaTeX},colorlinks,
    bookmarks,
    pdftitle={},
    pdfauthor={},
    pdfkeywords={},
    linkcolor=blue,                  % color of internal links
    citecolor=red,                   % color of links to bibliography
    filecolor=red,                   % color of file links
    urlcolor=cyan                    % color of external links
]{hyperref}

\pgfrealjobname{VTGirthSix}

\newcolumntype{C}[1]{>{\centering\arraybackslash}m{#1\linewidth}}

\newcommand{\doi}[1]{\href{http://dx.doi.org/#1}{\texttt{doi:#1}}}
\newcommand{\arxiv}[1]{\href{http://arxiv.org/abs/#1}{\texttt{arXiv:#1}}}

\newcommand{\G}{\ensuremath{\Gamma}}
\newcommand{\D}{\ensuremath{\mathbb{D}}}
\newcommand{\Z}{\ensuremath{\mathbb{Z}}}

\newtheorem{definition}{Definition}
\newtheorem{theorem}[definition]{Theorem}
\newtheorem{proposition}[definition]{Proposition}
\newtheorem{lemma}[definition]{Lemma}
\newtheorem{corollary}[definition]{Corollary}

\newcommand{\Aut}{\ensuremath{\operatorname{Aut}}}

\newcommand{\Sym}{\ensuremath{\mathbb{S}}}

\newcommand{\GP}{\ensuremath{\operatorname{GP}}}
\newcommand{\PGL}{\ensuremath{\operatorname{PGL}}}
\newcommand{\GL}{\ensuremath{\operatorname{GL}}}

\newcommand{\out}{\ensuremath{\operatorname{out}}}

\newcommand{\la}{\langle}
\newcommand{\ra}{\rangle}

\newcommand{\V}{\ensuremath{\mathcal{V}}}
\newcommand{\E}{\ensuremath{\mathcal{E}}}

\newcommand{\Cay}{\ensuremath{\operatorname{Cay}}}
\newcommand{\Tr}{\ensuremath{\operatorname{Tr}}}
\newcommand{\SDW}{\ensuremath{\operatorname{SDW}}}

\newcommand{\A}{\ensuremath{\mathcal{A}}}
\newcommand{\M}{\ensuremath{\mathcal{M}}}
\newcommand{\T}{\ensuremath{\mathcal{T}}}

\newcommand{\lr}{\leftrightarrow}

\newcommand{\mm}{\!-\!}
\newcommand{\pp}{\!+\!}

\newcommand{\grp}[2]{\ensuremath{\left\la #1 \; \middle| \; #2 \right\ra}}

\begin{document}

\title[Cubic vertex-transitive graphs of girth six]{Cubic vertex-transitive graphs of girth six}

\author[P.\ Poto\v{c}nik]{Primo\v{z} Poto\v{c}nik}
\address{Primo\v{z} Poto\v{c}nik,\newline
Faculty of Mathematics and Physics, University of Ljubljana, \newline Jadranska 19, SI-1000 Ljubljana, Slovenia;
\newline also affiliated with: \newline
%IAM, University of Primorska,\newline
% Muzejski trg 2, SI-6000 Koper, Slovenia;
% \newline and: \newline
 IMFM,
 Jadranska 19, SI-1000 Ljubljana, Slovenia.
 }
\email{primoz.potocnik@fmf.uni-lj.si}

\thanks{Supported in part by the Slovenian Research Agency, projects J1-1691, P1-0285, and P1-0294}

%\thanks{The first author gratefully acknowledges the supported by Slovenian Research Agency, projects L1--4292 and P1--0222.}

\author[J.\ Vidali]{Jano\v{s} Vidali}
\address{Jano\v{s} Vidali,\newline
Faculty of Mathematics and Physics, University of Ljubljana, \newline Jadranska 19, SI-1000 Ljubljana, Slovenia.
\newline also affiliated with: \newline
 IMFM,
 Jadranska 19, SI-1000 Ljubljana, Slovenia.
}
 \email{janos.vidali@fmf.uni-lj.si}


\begin{abstract}
In this paper,
a complete classification of finite simple cubic vertex-transitive graphs
of girth $6$ is obtained.
It is proved that every such graph,
with the exception of the Desargues graph on $20$ vertices,
is either a skeleton of a hexagonal tiling of the torus,
the skeleton of the truncation of an arc-transitive triangulation
of a closed hyperbolic surface,
or the truncation of a $6$-regular graph
with respect to an arc-transitive dihedral scheme.
Cubic vertex-transitive graphs of girth larger than $6$ are also discussed.
\end{abstract}

\maketitle

%%%%%%%%%%%%%%%%%%%%%%%%%%%%%%%%%%%%%%%%%%%%%%%%%%%%%%%%%

\section{Introduction}
\label{sec:intro}

Cubic vertex-transitive graph
are one of the oldest themes in algebraic graph theory,
appearing already in the classical work
of Foster~\cite{f32,fb88} and Tutte~\cite{t47},
and retaining the attention of the community until present times
(see, for example, the works of Coxeter, Frucht and Powers~\cite{cfp81},
Djokovi\'{c} and Miller~\cite{djm80},
Lorimer~\cite{l83}, Conder and Lorimer~\cite{cl89},
Glover and Maru\v{s}i\v{c}~\cite{gm07},
Poto\v{c}nik, Spiga and Verret~\cite{psv15},
Hua and Feng~\cite{hf11}, Spiga~\cite{s14},
%Marc~\cite{m17}
to name a few of the most influential papers).

The {\em girth} (the length of a shortest cycle)
is an important invariant of a graph
which appears in many well-known graph theoretical problems,
results and formulas.
In many cases,
requiring the graph to have small girth
%(or perhaps large girth)
severely restricts the structure of the graph.

%When restricted to vertex-transitive graphs,
%a graph having girth $g$ does not imply only
%that one cycle of length $g$ exists in the graph,
%but rather that every vertex lies on at least one cycle of girth $g$.
%It is intuitively clear that occurrence of small cycles
%implies that the sizes of the balls centred in a given vertex
%cannot grow very fast with the radius.
%It is not surprising that within the family
%of vertex-transitive graphs of small girth

Such a phenomenon can be observed
when one focuses to a family of graphs of small valence
possessing a high level of symmetry.
For example,
arc-transitive $4$-valent graphs of girth at most $4$
were characterised in~\cite{pw07}.
In the case of cubic graphs, even more work has been done.
The structure of cubic arc-transitive graphs of girth at most $7$ and $9$
was studied in~\cite{fn06} and~\cite{cn07}, respectively,
and those of girth $6$ were completely determined in~\cite{km09}.
By requiring more symmetry, some of these results can be pushed further;
for example, in~\cite{m91},
cubic $4$-arc-transitive graphs of girth at most $13$ were classified,
while in~\cite{pmk17},
locally $3$-transitive graphs of girth $4$ are considered.
Recently, two papers appeared where the condition on arc-transitivity
was relaxed (considerably!)~to vertex-transitivity;
namely as a byproduct of the results
proved independently in~\cite{ejs19} and~\cite{pv19},
all cubic vertex-transitive graphs of girth $5$ are known.
There are several, sometimes surprising,
applications of such classification results
(see, for example,~%
\cite{bkw20} for an application in the theory of abstract polytopes,~%
\cite{hikst19} for an application regarding the distinguishing number,
and~\cite{km19} for a connection with the question
of existence of odd automorphisms of graphs).

The purpose of this paper is to extend the above mentioned
classification of cubic vertex-transitive graphs of girth at most $5$
to a significantly more complex situation
of vertex-transitive cubic graphs of girth $6$.
There are three generic sources
of cubic vertex-transitive graphs of girth $6$:
hexagonal tessellations of the torus
with three hexagons meeting at each vertex
(that is, vertex-transitive maps on the torus of type $\{6,3\}$---%
note that all of them are vertex-transitive),
truncations of arc-transitive triangulations of hyperbolic surfaces
(that is,
truncations of arc-transitive maps of type $\{3,\ell\}$ with $\ell \ge 7$),
and truncations of $6$-valent graphs
admitting an arc-transitive group of automorphisms
whose vertex-stabilisers act on the neighbourhoods
either as a cyclic or as a dihedral group of degree $6$
(these objects were dubbed {\em arc-transitive dihedral schemes}
in~\cite{pv19}).
More formal definitions of dihedral schemes and maps will be given
in Sections~\ref{ssec:schemes} and~\ref{ssec:maps}, respectively.

The main result of this paper states that
with the exception of a single graph,
the famous Desargues graph on $20$ vertices
(that can also be defined as the generalised Petersen graph $\GP(10, 3)$),
every cubic vertex-transitive graph of girth $6$
arises in one of the above three ways.
In Theorem~\ref{thm:g6},
we refine this statement
by classifying the cubic vertex-transitive graphs of girth $6$
by their {\em signature},
which, roughly speaking,
encodes the distribution of girth cycles throughout the graph.

Let us make this more precise.
For an edge $e$ of a graph $\Gamma$,
let $\epsilon(e)$ denote the number of girth cycles containing the edge $e$.
Let $v$ be a vertex of $\Gamma$
and let $\{e_1, \ldots, e_k\}$ be the set of edges incident to $v$
ordered in such a way that
$\epsilon(e_1) \le \epsilon(e_2) \le \ldots \le \epsilon(e_k)$.
Following~\cite{pv19},
the $k$-tuple $(\epsilon(e_1), \epsilon(e_2), \ldots, \epsilon(e_k))$
is then called the {\em  signature} of $v$.
A graph $\Gamma$ is called {\em girth-regular}
provided that all of its vertices have the same signature
(and if in addition $\epsilon(e_1) = \ldots = \epsilon(e_k)$,
the graph is called girth-edge-regular; see~\cite{jkm18}).
The signature of a vertex is then called the signature of the graph.
Clearly, every vertex-transitive graph is also girth-regular.

We can now state the main result of the paper.
The exceptional graphs $\Psi_n$, $\Sigma_n$ and $\Delta_n$
appearing in the theorem below are defined in Section~\ref{ssec:special}.
%The main result of~\cite{pv19} was the classification
%of cubic girth-regular graphs of girth at most $5$.
%In this paper, we extend that classification
%to cubic vertex-transitive graphs of girth $6$.
%Note that part (a) of Theorem~\ref{thm:g6} below
%can also be seen as a refined classification of toroidal vertex-transitive
%maps of type $\{6,3\}$.

\begin{theorem} \label{thm:g6}
Let $\G$ be a connected  cubic graph.
Then $\G$ is vertex-transitive and has girth $6$
if and only if $\G$ is one of the following:
\begin{enumerate}[(a)]
\item the skeleton of a
%vertex-transitive
map of type $\{6, 3\}$ on a torus,
with signature
    \begin{itemize}
    \item $(8, 8, 8)$ for $\Psi_7$ (Heawood graph),
    \item $(6, 6, 6)$ for $\Psi_8$ (Möbius-Kantor graph),
    \item $(4, 5, 5)$ for $\Psi_9 \cong \Delta_3$,
    \item $(4, 4, 4)$ for $\Sigma_3$ (Pappus graph),
    \item $(3, 4, 5)$ for $\Psi_n$ with $n \ge 10$,
    \item $(2, 3, 3)$ for $\Delta_n$ and $\Sigma_n$ with $n \ge 4$, and
    \item $(2, 2, 2)$ otherwise;
    \end{itemize}
\item the skeleton of the truncation
of an arc-transitive map of type $\{3, \ell\}$ with $\ell \ge 7$,
with signature $(1, 1, 2)$;
\item the truncation of a $6$-regular graph $\hat{\G}$
with respect to an arc-transitive dihedral scheme,
with signature $(0, 1, 1)$; or
\item the Desargues graph with signature $(4, 4, 4)$.
\end{enumerate}
\end{theorem}

Note that all maps of type $\{6,3\}$ on the torus are known
and have been classified independently by several authors
(see for example~\cite{a73,hopiw12,k86,s97,t91})
and two recent surveys of some of these classifications
have appeared in~\cite{a20,adks20}.
It is not difficult to see that every
toroidal map of type $\{6,3\}$ is vertex-transitive.
In fact, as was shown in~\cite{ad09},
all of them are Cayley graphs on generalised dihedral groups.

As a byproduct of Theorem~\ref{thm:g6}, together with~\cite[Theorem~1.5]{pv19},
where cubic vertex-transitive graphs of girth $5$ are classified,
we obtain the following refinement
of the classification of maps of type $\{6,3\}$
(hexagonal tessellations) on the torus:

\begin{corollary}
Let $\Gamma$ be the skeleton of a map of type $\{6,3\}$ on the torus.
If $\Gamma$ has no cycles of length less than $6$,
then either $\Gamma$ is one of the graphs $\Psi_n$ with $n\ge 7$,
$\Sigma_n$ with $n\ge 3$, $\Delta_n$ with $n\ge 4$,
or the only $6$-cycles of $\Gamma$ are the face cycles.
\end{corollary}

In Section~\ref{sec:def},
the necessary definitions and auxiliary results are stated.
Section~\ref{sec:g6} is devoted to the proof of Theorem~\ref{thm:g6},
while in Section~\ref{sec:g7},
cubic vertex-transitive graphs of girth larger than $6$ are discussed.

\section{Definitions and notation}
\label{sec:def}

\subsection{Graphs}
\label{ssec:graphs}

Even though we are mainly interested in simple graphs,
it will prove convenient to allow graphs to have loops and parallel edges.
For this reason,
define a {\em graph} as a triple $(V,E,\partial)$,
where $V$ and $E$ are the {\em vertex-set}
and the {\em edge-set} of the graph,
and $\partial \colon E \to \{ X : X \subseteq V, |X| \le 2\}$
is a mapping that maps an edge to the set of its end-vertices.
If $|\partial(e)| = 1$, then $e$ is a {\em loop}.
Two edges $e$ and $e'$ are {\em parallel} if $\partial(e) = \partial(e')$.
Graphs with no loops and parallel edges are {\em simple}
and can be thought of in the usual manner as a pair $(V,\sim)$,
where $V$ is the vertex-set
and $\sim$ is an irreflexive symmetric adjacency relation on $V$.

The vertex-set and the edge-set of a graph $\Gamma$
are denoted by $\V(\Gamma)$ and $\E(\Gamma)$, respectively.
Further, we let each edge consist of two mutually inverse {\em arcs},
each of the two arcs having one of the end-vertices as its {\em tail}.
For an arc $s$, we denote its inverse by $s^{-1}$.
The {\em head} of an arc $s$ is defined as the tail of $s^{-1}$.
%If the graph has no loops,
%we may identify an arc with tail $v$ underlying edge $e$
%with the pair $(v,e)$.
The set of arcs of a graph $\G$ is denoted by $\A(\G)$,
and the set of the arcs with their tail being a specific vertex $u$
by $\out_\G(u)$.
The valence of a vertex $u$ is defined as
the cardinality of $\out_\G(u)$.

An {\em automorphism} of a graph $\Gamma:=(V,E,\partial)$
is a permutation $\alpha$ of $V \cup E$ preserving $V$ and $E$
and satisfying $\partial(e^\alpha) = \{u^\alpha, v^\alpha\}$
for every edge $e\in E$ such that $\partial(e) = \{u, v\}$.
As usual,
we denote the group of all automorphisms of $\Gamma$ by $\Aut(\Gamma)$.

If $\Gamma$ is a simple graph,
then each automorphism of $\Gamma$
is uniquely determined by its action on $\V(\Gamma)$,
so we may think of it as an adjacency-preserving permutation of $\V(\Gamma)$.
Observe that every automorphism of $\Gamma$
induces a permutation of $\A(\Gamma)$.
If $G$ is a subgroup of $\Gamma$ that acts transitively on vertices,
edges or arcs of $\Gamma$,
then we say that $\Gamma$ is $G$-vertex-transitive,
$G$-edge-transitive or $G$-arc-transitive, respectively,
with the prefix $G$ typically omitted if $G=\Aut(\Gamma)$.

\subsection{Dihedral schemes}
\label{ssec:schemes}

Following~\cite{pv19},
a {\em dihedral scheme} on a graph $\G$
is an irreflexive symmetric relation $\lr$ on $\A(\G)$ such that
the simple graph $(\A(\G),\lr)$ is a $2$-regular graph
each of whose connected components
is the set $\out_\G(u)$ for some $u\in \V(\G)$.
Intuitively,
we may think of a dihedral scheme as an arrangement
of arcs around each vertex into a non-oriented cycle.
The group of all automorphisms of $\G$ that preserve the relation $\lr$
will be denoted by $\Aut(\G,\lr)$
and the dihedral scheme $\lr$ is said to be {\em arc-transitive}
if $\Aut(\G,\lr)$ acts transitively on $\A(\G)$.

Given a dihedral scheme $\lr$ on a graph $\G$,
we define the {\em truncation of $\G$ with respect to $\lr$}
as the simple graph $\Tr(\G,\lr)$ whose vertex set is $\A(\G)$,
with two arcs $s,t\in \A(\G)$ adjacent in $\Tr(G,\lr)$
if either $t \lr s$ or $t$ and $s$ are inverse to each other
(see the example in Figure~\ref{fig:trunc}).
%Note that $\Tr(\G, \lr)$ is always a cubic graph
%and is connected if and only if $\G$ is connected.
Observe that $\Aut(\G,\lr)$
acts as a group of automorphisms of $\Tr(\G,\lr)$,
implying that $\Tr(\G,\lr)$ is vertex-transitive
whenever the dihedral scheme $\lr$ is arc-transitive.

\begin{figure}[t]
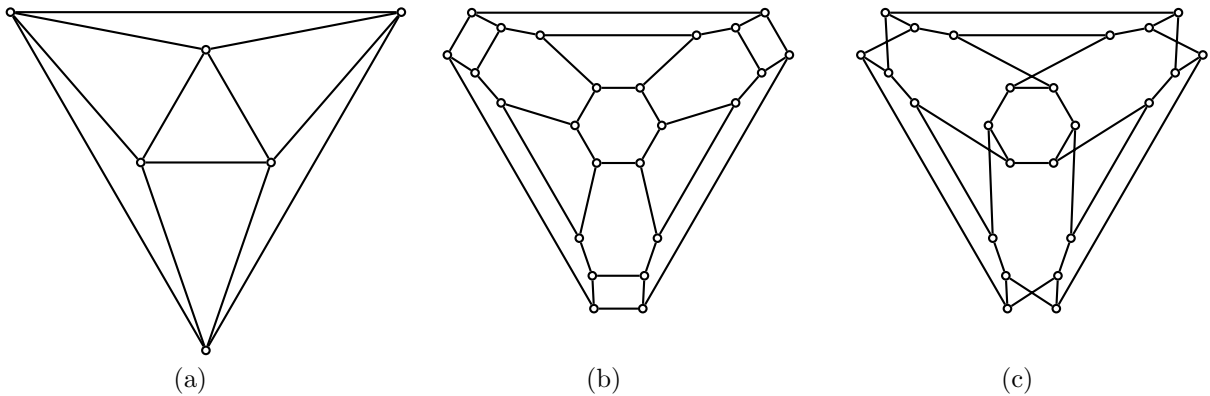

\makebox[\textwidth][c]{
\begin{tabular}{C{0.31}C{0.31}C{0.31}}
\leavevmode%
\beginpgfgraphicnamed{fig-w3}
\input{tikz/w3.tikz}
\endpgfgraphicnamed
&
\leavevmode%
\beginpgfgraphicnamed{fig-w3-trunc1}
\input{tikz/w3-trunc1.tikz}
\endpgfgraphicnamed
&
\leavevmode%
\beginpgfgraphicnamed{fig-w3-trunc2}
\input{tikz/w3-trunc2.tikz}
\endpgfgraphicnamed
\\
(a) & (b) & (c)
\end{tabular}
}
\caption{(a) The octahedral graph, a $4$-regular graph.
(b) The truncation of the octahedral graph with respect to the dihedral scheme
obtained by considering the drawing (a) as a map (i.e., an octahedron).
(c) The truncation of the octahedral graph
with respect to a different dihedral scheme.
Note that in both truncations,
vertices of the graph in (a) have been replaced by $4$-cycles.
}
\label{fig:trunc}
\end{figure}

As was proved in~\cite[Lemma~3.5]{pv19},
arc-transitive dihedral schemes
all arise in the following group theoretical setting.
Let $\Gamma$ be a $G$-arc-transitive graph (possibly with parallel edges)
such that the permutation group $G_u^{\out_\G(u)}$
induced by the action of the vertex-stabiliser $G_u$ on $\out_\Gamma(u)$
is permutation isomorphic to the transitive action of $\D_d$, $\Z_d$,
or (if $d$ is even) $\D_{\frac{d}{2}}$ on $d$ vertices
(here, the symbol $\D_d$ denotes the dihedral group of order $2d$
acting naturally on $d$ points,
while $\Z_d$ is the cyclic group acting transitively on $d$ points).
Fix a vertex $u$ of $\Gamma$ and choose
an adjacency relation $\lr_u$ on $\out_\G(u)$
preserved by $G_u^{\out_\G(u)}$ in such a way
that $(\out_\G(u),\lr_u)$ is a cycle
(note that the assumption on $G_u^{\out_\G(u)}$
implies that such a relation exists).
For every $v\in \V(\G)$,
choose an element $g_v\in G$ such that $v^{g_v} = u$,
and let $\lr_v$ be the relation on ${\out_\G(v)}$
defined by $s \lr_v t$ if and only if $s^{g_v} \lr_u t^{g_v}$.
Then the union $\lr$ of all $\lr_u$ for $u\in V(\G)$
is a dihedral scheme invariant under $G$.

We would like to point out that an equivalent definition of dihedral schemes
and a generalization of the corresponding truncations
was given recently in~\cite{ejs19} (see also~\cite{ej12}).
To obtain a truncation as defined above,
the graph $\Upsilon$ in the definition of the generalised truncation
in~\cite[Section~2]{ejs19} needs to be a cycle.

\subsection{Maps}
\label{ssec:maps}

Topologically,
a map is an embedding of a finite connected graph onto a closed surface
in such a way that when the graph is removed from the surface,
the connected components (called {\em faces})
of what remains are homeomorphic to open disks
whose closures are closed disks.

There are several ways to describe a map combinatorially,
one way being by specifying a set of walks in the graph
that represent the boundaries of the faces of the map.
More precisely,
let $\G$ be a connected graph and let $\T$ be a set of closed walks in $\G$
such that every edge of $\G$ belongs to precisely two walks in $\T$.
We will also require that every edge is traversed
at most once by every walk in $\T$,
even though in the literature often this is not required
(the maps that satisfy our additional conditions
are then sometimes called {\em polyhedral}).
For two arcs $s$ and $t$ with a common tail,
write $s\lr t$ if and only if the underlying edges of $s$ and $t$
are two consecutive edges on a walk in $\T$.
If $\lr$ is a dihedral scheme,
then $(\G,\T)$ is a map with {\em skeleton} $\G$ and {\em face walks} $\T$.
The topological map can then be reconstructed from such a pair $(\G,\T)$
by thinking of the graph as a $1$-dimensional CW complex
and then gluing closed disks along its boundary
homeomorphically to the closed curves in $\G$
represented by elements of $\T$.
The resulting topological space is then a closed surface,
which can be either orientable or non-orientable.

An automorphism of a map $(\G,\T)$
is an automorphism of $\G$ that preserves the set $\T$.
Note that such an automorphism extends
to a homeomorphism of the resulting surface preserving the embedded graph.
The map is called vertex-, edge- or arc-transitive,
provided that its automorphism group acts transitively
on the vertices, edges or arcs of the underlying graph $\G$.
If the graph $\G$ is $k$-regular
and all the closed walks in $\T$ have length $\ell$,
then the map $(\G,\T)$ is said to be of {\em type} $\{\ell,k\}$.

There are two ways in which maps enter the classification
of cubic vertex-transitive graphs of girth $6$.
The first is when the skeleton of a map is a cubic graph
and the faces form the girth cycles,
that is,
when the map has type $\{6,3\}$
and it contains no shorter cycles than the face walks
(in this case, the face walks are cycles,
so we may refer to them as {\em face cycles}).
By computing the genus of the underlying surface using Euler's formula,
one sees that such a map resides either on the Klein bottle or on the torus.
However, as was shown in~\cite{w06},
there are no vertex-transitive maps of type $\{6,3\}$ and girth more than $4$
on the Klein bottle.
On the other hand,
there are numerous toroidal vertex-transitive maps
of type $\{6,3\}$ and girth $6$,
and all of them are vertex-transitive.
As mentioned in Section~\ref{sec:intro},
toroidal maps of type $\{6,3\}$ have been extensively studied
from different angles and have been independently classified several times
(see~\cite{adks20} or~\cite{a20} for recent surveys).

The second way in which maps yield
cubic vertex-transitive graphs of girth $6$
is by taking (the skeleton of) the truncation of an arc-transitive map
of type $\{3,\ell\}$ with $\ell \ge 7$.
Here, the truncation of a map $(\G, \T)$ has the usual meaning --
note that its skeleton is the truncation of the underlying graph $\G$
with respect to the dihedral scheme appearing in the definition of the map.

\subsection{Three special families of toroidal graphs}
\label{ssec:special}

In this section,
we define the graphs $\Psi_n$, $\Sigma_n$ and $\Delta_n$
appearing in Theorem~\ref{thm:g6}.
They are all skeletons of toroidal maps of type $\{6,3\}$,
and unlike other toroidal maps of type $\{6,3\}$,
they possess $6$-cycles other than the face cycles (and no shorter cycles).
We will introduce them as Cayley graphs.
Recall that a Cayley graph $\Cay(G,S)$ on a group $G$
with the connection set $S$, $S\subseteq G\setminus \{1_G\}$, $S=S^{-1}$,
is a simple graph whose vertices are elements of $G$,
with $g,h\in G$ adjacent whenever $gh^{-1} \in S$.

Since a Cayley graph $\Cay(G,S)$ is vertex-transitive,
it is also girth-regular.
One can determine its girth $g$
by finding the length of the shortest nonempty sequence
$(\alpha_1, \alpha_2, \dots, \alpha_g)$
with $\alpha_i \in S$ ($1 \le i \le g$)
such that $\alpha_i \alpha_{i+1} \ne 1_G$ ($1 \le i \le g-1$)
and $\alpha_1 \alpha_2 \cdots \alpha_g = 1_G$.
The signature can then be determined
by identifying all such sequences of length $g$
and counting how many times each element of the connection set
appears as the first symbol in these sequences.

In what follows,
let $\D_d$ denote the dihedral group of order $2d$
acting naturally on $d$ points,
and let $\Z_d$ be a cyclic group of order $d$ acting on $d$ points.
% and let $\Sym_d$ be the symmetric group of degree $d$.
For dihedral groups, we will use the presentation
$\D_n = \grp{\rho, \tau}{\rho^n, \tau^2, (\rho \tau)^2}$
-- i.e., $\rho$ represents a unit rotation,
while $\tau$ represents a reflection of the points around some axis.
For brevity, we denote $\rho_i = \rho^i$ and $\tau_i = \rho^i \tau$,
where indices are modulo $n$.
It is easy to see that for all integers $i, j$,
we have $\rho_i \rho_j = \rho_{i+j}$, $\rho_i \tau_j = \tau_{i+j}$,
$\tau_i \rho_j = \tau_{i-j}$ and $\tau_i \tau_j = \rho_{i-j}$.
For the direct product $\D_n \times \Z_3$,
we abbreviate its member $(\alpha_i, u)$ $(\alpha \in \{\rho, \tau\})$
as $\alpha_i^0$, $\alpha_i^+$ or $\alpha_i^-$ if $u = 0, 1, 2$, respectively.

For a positive integer $n$,
we define the graph
$\Delta_n = \Cay(\D_{3n}, \{\tau_0, \tau_k, \tau_n\})$ of order $6n$,
where $k = 3/\gcd(3, n)$.
The graph $\Delta_n$ is vertex-transitive with girth $6$ for all $n \ge 3$,
with signature $(4, 5, 5)$ if $n = 3$ and $(2, 3, 3)$ otherwise.
For all $n \ge 1$,
the graph $\Delta_n$ admits an embedding onto a torus
with $3n$ hexagonal faces such that the consecutive arcs on each face
correspond to the generators $\tau_0, \tau_k, \tau_n, \tau_0, \tau_k, \tau_n$.
The graphs $\Delta_4$ and $\Delta_5$ are shown
in Figure~\ref{fig:g6c3even}(b) and Figure~\ref{fig:g6c3odd}(b).

For a positive integer $n$,
we next define the graph
$\Sigma_n = \Cay(\D_n \times \Z_3, \{\tau_1^0, \tau_0^+,$ $\tau_0^-\})$
of order $6n$.
The graph $\Sigma_n$ is vertex-transitive with girth $6$ for all $n \ge 3$.
The graph $\Sigma_3$ is the Pappus graph with signature $(4, 4, 4)$,
while for $n \ge 4$, $\Sigma_n$ has signature $(2, 3, 3)$,
For all $n \ge 1$,
the graph $\Sigma_n$ admits an embedding onto a torus
with $3n$ hexagonal faces such that the consecutive arcs on each face
correspond to the generators
$\tau_0^+, \tau_0^+, \tau_1^0, \tau_0^-, \tau_0^-, \tau_1^0$.
The graphs $\Sigma_4$ and $\Sigma_5$ are shown
in Figure~\ref{fig:g6c3even}(a) and Figure~\ref{fig:g6c3odd}(a).
Note also that the graph $\Sigma_n$ is isomorphic to the so-called
{\em split depleted wreath graph} $\SDW(n, 3)$
(cf.~\cite{w02} and~\cite[Construction~11]{psv13})
defined to have the vertex-set $\Z_n \times \Z_3 \times \Z_2$
and edges of two types:
$\{(i, u, 0), (i, u \pm 1, 1)\}$ and $\{(i, u, 1), (i+1, u, 0)\}$,
for all $i \in \Z_n$ and $u \in \Z_3$.
An isomorphism between $\Sigma_n$ and $\SDW(n, 3)$
can be chosen so that it maps
$(\rho_i, u) \mapsto (i, u, 1)$ and $(\tau_i, u) \mapsto (i, u, 0)$
for every $i \in \Z_n$ and $u \in \Z_3$
(see Figure~\ref{fig:SDW}).

\begin{figure}[t]
\leavevmode%
\beginpgfgraphicnamed{fig-sdw}
\input{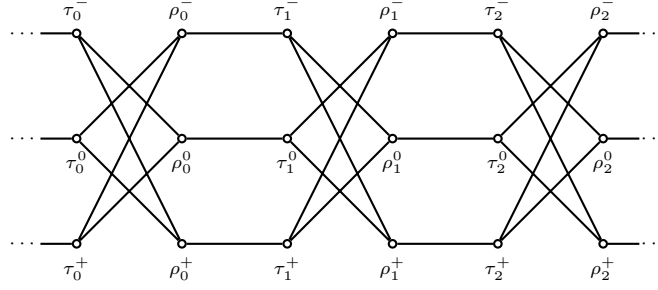}
\endpgfgraphicnamed
\caption{A section of $\Sigma_n \cong \SDW(n, 3)$.}
\label{fig:SDW}
\end{figure}

Finally, for a positive integer $n$ we define the graph
$\Psi_n = \Cay(\D_n, \{\tau_0, \tau_1, \tau_3\})$ of order $2n$.
The graph $\Psi_n$ is vertex-transitive with girth $6$ for all $n \ge 7$.
The graphs $\Psi_7$ and $\Psi_8$ are the Heawood and Möbius-Kantor graphs
with signatures $(8, 8, 8)$ and $(6, 6, 6)$, respectively.
The graph $\Psi_9$ is isomorphic to $\Delta_3$ with signature $(4, 5, 5)$.
For $n \ge 10$, $\Psi_n$ has signature $(3, 4, 5)$.
For all $n \ge 1$,
the graph $\Psi_n$ admits an embedding onto a torus
with $n$ hexagonal faces such that the consecutive arcs on each face
correspond to the generators $\tau_0, \tau_1, \tau_3, \tau_0, \tau_1, \tau_3$.
The graph $\Psi_{10}$ is shown in Figure~\ref{fig:psi}.

\begin{figure}[t]
\leavevmode%
\beginpgfgraphicnamed{fig-g6-c3o}
\input{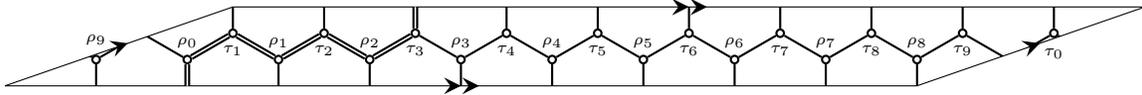}
\endpgfgraphicnamed
\caption{The graph $\Psi_{10}$ embedded on a torus.
The double edges show one of the $6$-cycles
which do not correspond to a face of the embedding.}
\label{fig:psi}
\end{figure}

Let us now determine the full groups of automorphisms of these graphs.
We first give a lemma which will be useful in determining their automorphisms.

\begin{lemma} \label{lem:autcay}
Let $i, j, k, m$ be distinct integers
with $0 \le i, j, k < m$, $|j-k| \ne m/2$ and $\gcd(i, j, k, m) = 1$,
and let $\G = \Cay(\D_m, \{\tau_i, \tau_j, \tau_k\})$.
Define $A_h$ ($h \in \{i, j, k\}$) to be the set of arcs of $\G$
corresponding to the generator $\tau_h$.
Suppose that $\varphi$ is an automorphism of the graph $\G$
that fixes the set $A_i$.
Then either $\varphi$ fixes the sets $A_j$ and $A_k$, or swaps them.
\end{lemma}

\begin{proof}
First, we note that the graph $\G$ contains $6$-cycles
such that its consecutive arcs correspond to the generators
$\tau_i, \tau_j, \tau_k, \tau_i, \tau_j, \tau_k$.
Let $C$ be such a cycle.
As $\tau_i \tau_j \tau_k \tau_i \tau_k \tau_j = \rho_{2j-2k} \ne \rho_0$,
the consecutive arcs of $C^\varphi$
must correspond to the same generators as those of $C$
-- thus, all arcs of $C$ corresponding to $\tau_j$ (respectively $\tau_k$)
map to arcs of $C^\varphi$ all corresponding to $\tau_j$,
or all corresponding to $\tau_k$.
As each of these arcs also lies on another $6$-cycle $C'$ whose consecutive arcs
correspond to the same generators as those of $C$,
it follows that $\varphi$ acts in the same way on the arcs of $C'$.
Since the connection set generates the group $\D_m$,
the graph $\G$ is connected,
therefore this is true for all such $6$-cycles.
We thus conclude that either $\varphi$ fixes the sets $A_j$ and $A_k$,
or swaps them.
\end{proof}

We first deal with the graphs $\Delta_n$ with $n \ge 3$.
% The full automorphism group of both $\Delta_1$ and $\Delta_2$
% is the direct product $\Sym_4 \times \Z_2$ of order $48$
% (i.e., the full cube group).
Recall that a Cayley graph $\Cay(G, S)$
is called a {\em graphical regular representation} of the group $G$
if its full group of automorphisms is isomorphic to $G$.

\begin{proposition} \label{prop:autdelta}
Let $n \ge 3$.
The full group of automorphisms of the graph $\Delta_n$
is isomorphic to the dihedral group $\D_{3n}$ of order $6n$,
i.e., $\Delta_n$ is a graphical regular representation of $\D_{3n}$.
\end{proposition}

\begin{proof}
As noted above, the graph $\Delta_n = (V, E, \partial)$ contains $6$-cycles
such that their consecutive arcs correspond to the generators
$\tau_0, \tau_k, \tau_n, \tau_0, \tau_k, \tau_n$,
where $k = 3/\gcd(3, n)$
-- each arc lies on two such $6$-cycles.
Furthermore, there are $6$-cycles in $\Delta_n$
whose consecutive arcs correspond to the generators
$\tau_0, \tau_n, \tau_0, \tau_n, \tau_0, \tau_n$.
Since $\gcd(3n, n-k) = 1$,
the arcs corresponding to $\tau_k$ and $\tau_n$
form a Hamiltonian cycle $H$ of $\Delta_n$.
From the definition it follows that the group $\D_{3n}$
acts regularly on the vertices of $\Delta_n$ by right-multiplication,
which induces the natural action of $\D_{3n}$ on the edges of $H$
corresponding to $\tau_k$ (respectively $\tau_n$).
In particular,
the automorphisms from $\D_{3n}$ are precisely those
which fix the sets $A_0$, $A_k$ and $A_n$ defined as in Lemma~\ref{lem:autcay}.
We will show that $\Delta_n$ does not admit any other automorphism.

If $n \ge 4$,
there are no other $6$-cycles in $\Delta_n$ other than the ones described above
-- the arcs corresponding to $\tau_k$ thus lie on two $6$-cycles,
while the arcs corresponding to $\tau_0$ or $\tau_n$ lie on three $6$-cycles.
Every automorphism of $\Delta_n$ thus fixes the set $A_k$,
and by Lemma~\ref{lem:autcay},
either fixes or swaps the sets $A_0$ and $A_n$.
If $n$ is not a multiple of $3$,
then $k = 3$ and the arcs corresponding to $\tau_0$ and $\tau_k = \tau_3$
form three $2n$-cycles.
If, on the other hand, $n$ is a multiple of $3$,
then $k = 1$ and the arcs corresponding to $\tau_0$ and $\tau_k = \tau_1$
form another Hamiltonian cycle $H'$ of $\Delta_n$.
Now, an edge $e$ with $\partial(e) = \{\rho_i, \tau_i\}$
(i.e., corresponding to $\tau_0$)
along with the shortest path on $H$ between $\rho_i$ and $\tau_i$ forms a cycle
$\rho_i \tau_{i+n} \rho_{i+n-1} \tau_{i+2n-1} \dots \rho_{i-n} \tau_i$
of length $2n+2$
since $(n+1) \cdot n - n \cdot 1 \equiv n^2 \equiv 0 \pmod{3n}$.
However, an edge $e'$ with $\partial(e') = \{\rho_i, \tau_{i+n}\}$
(i.e., corresponding to $\tau_n$)
along with the shortest path on $H'$ between $\rho_i$ and $\tau_{i+n}$
forms a $2n$-cycle
$\rho_i \tau_{i+1} \rho_{i+1} \tau_{i+2} \dots \rho_{i+n-1} \tau_{i+n}$.
Therefore,
no automorphism of $\Delta_n$ swaps the Hamiltonian cycles $H$ and $H'$.
In either case it then follows that
no automorphism of $\Delta_n$ swaps the sets $A_0$ and $A_n$.

Consider now the case $n = 3$ -- we then have $k = 1$.
Besides the $6$-cycles described above,
the graph $\Delta_3$ also contains $6$-cycles
whose consecutive arcs correspond to the generators
$\tau_0, \tau_1, \tau_0, \tau_1, \tau_3, \tau_1$.
The arcs corresponding to $\tau_n = \tau_3$ thus lie on four $6$-cycles,
while the arcs corresponding to $\tau_0$ or $\tau_k = \tau_1$
lie on five $6$-cycles.
Every automorphism of $\Delta_3$ thus fixes the set $A_3$,
and by Lemma~\ref{lem:autcay},
either fixes or swaps the sets $A_0$ and $A_1$.
But since sequences of arcs corresponding to the generators
$\tau_1, \tau_0, \tau_1, \tau_0, \tau_3, \tau_0$ do not form $6$-cycles,
it follows that no automorphism of $\Delta_3$ swaps the sets $A_0$ and $A_1$.
We thus conclude that for every $n \ge 3$,
the full automorphism group of $\Delta_n$ is isomorphic to the group $\D_{3n}$,
so $\Delta_n$ is its graphical regular representation.
\end{proof}

Let us now consider the graphs $\Sigma_n$.
The full automorphism group of $\Sigma_3$ has order $216$.
The following lemma deals with the case when $n \ge 4$.

\begin{proposition} \label{prop:autsigma}
Let $n \ge 4$.
The full group of automorphisms of the graph $\Sigma_n$
is isomorphic to the direct product $\D_n \times \D_3$ of order $12n$.
\end{proposition}

\begin{proof}
As noted above, the graph $\Sigma_n = (V, E, \partial)$ contains $6$-cycles
such that their consecutive arcs correspond to the generators
$\tau_0^+, \tau_0^+, \tau_1^0, \tau_0^-, \tau_0^-, \tau_1^0$
-- each arc lies on two such $6$-cycles.
Furthermore, there are $6$-cycles in $\Sigma_n$
whose consecutive arcs all correspond to the generator $\tau_0^+$
(or, equivalently, its inverse $\tau_0^-$).
We denote the set of such $6$-cycles by $C$.
Since $n \ge 4$, there are no other $6$-cycles,
so the arcs corresponding to $\tau_1^0$ lie on two $6$-cycles,
while the arcs corresponding to $\tau_0^+$ or $\tau_0^-$
lie on three $6$-cycles.
Every automorphism of $\Sigma_n$ thus fixes the set $C$,
as well as the set $E_1^0$ of edges corresponding to $\tau_1^0$.

We may thus define a graph $\Lambda = (C, E_1^0, \partial')$
with $\partial'(e) = \{c, c'\} \subset C$
such that $\partial(e) = \{u, v\}$ with $u \in c$, $v \in c'$
(i.e., the endpoints of $e$ in $\Lambda$ are the cycles of $C$
containing the endpoints of $e$ in $\Sigma_n$).
Furthermore,
we define a dihedral scheme $\lr$ on $\Lambda$ by letting $s \lr t$
whenever the tails of $s$ and $t$ in $\Sigma_n$ are adjacent.
Note that the graph $\Lambda$ is a $n$-cycle with tripled edges,
and that $\Sigma_n$ is isomorphic to $\Tr(\Lambda, \lr)$.

We now claim that
$\Aut(\Sigma_n) \cong \Aut(\Lambda, \lr) \cong \D_n \times \D_3$.
Indeed,
the left factor in the direct product
acts naturally on the vertices of $\Lambda$ and therefore on the cycles in $C$,
while the right factor
acts naturally on the sets of parallel edges of $\Lambda$
(while preserving the dihedral scheme $\lr$)
and therefore the sets of edges of $\Sigma_n$
connecting the same two cycles of $C$;
the two actions commute.
This covers every automorphism of $\Sigma_n$
which fixes the sets $C$ and $E^1_0$,
and these automorphisms then form the full automorphism group of $\Sigma_n$.
\end{proof}

Finally, we consider the graphs $\Psi_n$.
The graphs $\Psi_7$ and $\Psi_8$
have full automorphism groups $\PGL(2, 7)$ of order $336$
and $\GL(2,3) \rtimes \Z_2$ of order $96$
(with the non-identity element of $\Z_2$ acting as a composition
of the transposition and the inverse on $\GL(2,3)$),
respectively.
The following lemma deals with the case when $n \ge 9$.

\begin{proposition} \label{prop:autpsi}
Let $n \ge 9$.
The full group of automorphisms of the graph $\Psi_n$
is isomorphic to the dihedral group $\D_n$ of order $2n$,
i.e., $\Psi_n$ is a graphical regular representation of $\D_n$.
\end{proposition}

\begin{proof}
Since $\Psi_9 \cong \Delta_3$,
it follows from Proposition~\ref{prop:autdelta}
that the full automorphism group of the graph $\Psi_9$
is isomorphic to $\D_9$.
In the remainder of the proof we will thus assume $n \ge 10$.

Let $G$ be the full automorphism group of $\Psi_n$.
As noted above,
the signature of the graph $\Psi_n$ with $n \ge 10$ is $(3, 4, 5)$
-- since it consists of distinct integers,
it follows that the stabilizer $G_u$ of a vertex $u$ of $\Psi_n$
acts trivially on its neighbours,
and by connectivity it follows that $G_u$ is trivial.
From the definition it follows that the group $\D_n$
acts regularly on the vertices of $\Psi_n$ by right-multiplication,
so the full automorphism group $G$ is isomorphic to $\D_n$.
We can thus conclude that the graph $\Psi_n$
is a graphical regular representation of $\D_n$ for all $n \ge 9$.
\end{proof}

\subsection{Auxiliary results}

We will repeat some results from~\cite{pv19}
that will be used in the following section.
%For definitions of the truncation, dihedral schemes and maps,
%see~\cite[\S3]{pv19}.

\begin{lemma} \label{lem:eventriineq}
{\rm (\cite[Lemma~3.1]{pv19})}
\pushQED{\qed}
If $(a, b, c)$ is the signature
of a cubic girth-regular graph $\G$ of girth $g$, then:
\begin{enumerate}
\item \label{lem:eventriineq:1} $a+b+c$ is even,
\item \label{lem:eventriineq:2} $a+b \ge c$, and
\item \label{lem:eventriineq:3} if $a \ge 1$ and $c = a+b$,
then $g$ is even.
\qedhere
\end{enumerate}
\popQED
\end{lemma}

\begin{lemma}
\label{lem:azero}
{\rm (\cite[Lemma~3.2]{pv19})}
If the signature of a cubic girth-regular graph is $(0, b, c)$,
then $b = c = 1$.
\qed
\end{lemma}

\begin{lemma} \label{lem:ab}
{\rm (\cite[Lemma~3.4]{pv19})}
Let $\G$ be a cubic girth-regular graph of girth $g$
with signature $(a, b, c)$.
Let $m = 2^{\lfloor g/2 \rfloor - 1}$.
Then $a \ge c - m$ and $b \le a - c + 2m$.
\qed
\end{lemma}

\begin{theorem} \label{thm:azero}
{\rm (\cite[Theorem~3.6]{pv19})}
If $\G$ is a simple cubic girth-regular graph of girth $g$
with signature $(0, 1, 1)$,
then $\G$ is isomorphic to the truncation of a
$g$-regular graph $\Lambda$ (possibly with parallel edges)
with respect to a dihedral scheme $\lr$.
Moreover, if $\G$ is vertex-transitive,
then the dihedral scheme $\lr$ is arc-transitive.
\qed
\end{theorem}

\begin{theorem}
\label{thm:trieq1}
{\rm (\cite[Theorem~3.14]{pv19})}
Let $\G$ be a simple connected cubic girth-regular graph
of girth $g$ with $n$ vertices and signature $(1, 1, 2)$.
Then $g$ is even and $\G$ is the truncation of some map $\M$
with face cycles of length $g/2$.
In particular, $g/2$ divides $n$.
Moreover, if $\G$ is vertex-transitive,
then $\M$ is an arc-transitive map of type $\{g/2, \ell\}$
for some $\ell > g$.
\qed
\end{theorem}

\begin{theorem} \label{thm:abc2}
{\rm (\cite[Theorem~3.11]{pv19})}
Let $\G$ be a connected simple cubic girth-regular graph
of girth $g$ with $n$ vertices and signature $(2, 2, 2)$.
Then $g$ divides $3n$ and $\G$ is the skeleton of a map of type $\{g, 3\}$
embedded on a surface with Euler characteristic
$$\chi = n \left(\frac{3}{g} - \frac{1}{2} \right).$$
Moreover, every automorphism of $\G$ extends to an automorphism of the map.
In particular, if $\G$ is vertex-transitive, so is the map.
\qed
\end{theorem}

\begin{theorem} \label{thm:cmax}
{\rm (\cite[Theorems~1.2,~1.3,~1.4]{pv19})}
\pushQED{\qed}
Let $\G$ be a simple connected cubic girth-regular graph
with signature $(a, b, c)$.
Then $c \le 2^{\lfloor g/2 \rfloor}$,
with equality implying that $a = b = c$ and $\G$ is one of the following:
\begin{enumerate}[(a)]
\item the complete graph $K_4$ of girth $g = 3$,
\item the complete bipartite graph $K_{3,3}$ of girth $g = 4$,
\item the Petersen graph of girth $g = 5$,
\item the Heawood graph of girth $g = 6$,
\item the Tutte-Coxeter graph of girth $g = 8$, or
\item the Tutte $12$-cage of girth $g = 12$.
\qedhere
\end{enumerate}
\popQED
\end{theorem}

% TODO: change numbering if appropriate
\section[Proof of Theorem 1]{Proof of Theorem~\ref{thm:g6}}
\label{sec:g6}

This section contains the proof of Theorem~\ref{thm:g6}.
For cubic vertex-transitive graphs of girth $6$,
Lemmas~\ref{lem:eventriineq} and~\ref{lem:ab}
and Theorems~\ref{thm:azero} and~\ref{thm:cmax}
imply that there are $27$ possible signatures $(a, b, c)$.
In particular, $c \le 8$ must hold.
As we will see, only $9$ of these signatures actually occur.

Let $\G$ be a cubic vertex-transitive graph of girth $6$.
For an arc $uv$ of $\G$,
let $P(uv)$ be the partition
of the set of vertices at distance $2$ from $u$ that are not adjacent to $v$
into two sets according to their common neighbour with $u$
-- i.e., $P(uv)$ contains two sets with two vertices each.
We also define $T(uv)$ as the multiset of two multisets
containing numbers which, for each vertex of $P(uv)$,
tell how many neighbours it has among the vertices of $P(vu)$
(see Figure~\ref{fig:g6}(c) for an example).
We refer to the ordered pair $(T(uv), T(vu))$
as the {\em type} of the arc $uv$.
Clearly, $\sum \bigcup T(uv) = \sum \bigcup T(vu)$
equals the number of $6$-cycles the arc $uv$ lies on.
Vertex-transitivity of $\G$ implies that
for every arc $uv$ with type $(R, S)$,
there is an arc $uw$ with inverse type $(S, R)$.
Since the valency of each vertex is odd,
there must exist an arc with type $(R, R)$ for some $R$.
Such an arc is said to have {\em symmetric type}.
Thus, either all three arcs with tail $u$ have symmetric types,
or one has symmetric type
and the other two have mutually inverse asymmetric types.

Theorems~\ref{thm:azero},~\ref{thm:trieq1} and~\ref{thm:cmax}
already deal with signatures $(0, 1, 1)$, $(1, 1, 2)$ and $(a, b, 8)$,
respectively.
We will now consider each remaining signature $(a, b, c)$
in decreasing order of $c$.
In the following lemmas,
we will start with vertices
$u_0, u_1, v_{00}, v_{01}, v_{10},$ $v_{11},
w_{000}, w_{001}, w_{010}, w_{011}, w_{100}, w_{101}, w_{110}, w_{111}$
and edges as shown in Figure~\ref{fig:g6}(a),
and then add new vertices and edges to complete the graphs
or arrive at a contradiction.
Note that girth $6$ implies that for each arc $w_{hij} w_{h'i'j'}$
($h, h', i, i', j, j' \in \{0, 1\}$),
we must have $h \ne h'$,
and for each $2$-path $w_{hij} w_{h'i'j'} w_{hi''j''}$
($h, h', i, i', i'', j, j', j'' \in \{0, 1\}$),
we must have $i \ne i''$.
In particular,
there is no $6$-cycle containing only vertices $w_{hij}$
($h, i, j \in \{0, 1\}$).

\begin{figure}[t]
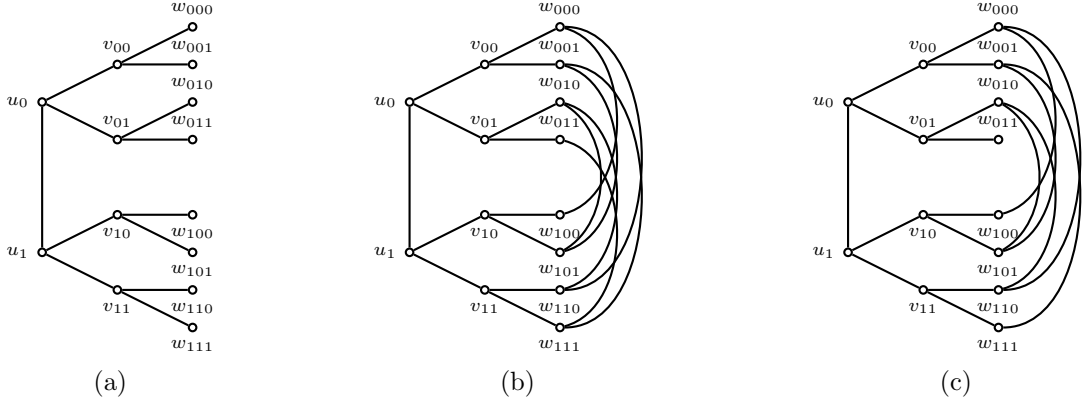

\makebox[\textwidth][c]{
\begin{tabular}{C{0.28}C{0.33}C{0.33}}
\leavevmode%
\beginpgfgraphicnamed{fig-g6}
\input{tikz/g6.tikz}
\endpgfgraphicnamed
&
\leavevmode%
\beginpgfgraphicnamed{fig-g6-c7}
\input{tikz/g6-c7.tikz}
\endpgfgraphicnamed
&
\leavevmode%
\beginpgfgraphicnamed{fig-g6-c6a}
\input{tikz/g6-c6a.tikz}
\endpgfgraphicnamed
\\
(a) & (b) & (c)
\end{tabular}
}
\caption{Constructing a graph of girth $6$.
The general setting is shown in (a).
(b) and (c) show the cases when the arc $u_0 u_1$
lies on seven $6$-cycles,
and six $6$-cycles with $T(u_0 u_1) = \{\{2, 2\}, \{2, 0\}\}$,
respectively.
A contradiction arises in both cases.
}
\label{fig:g6}
\end{figure}

\begin{lemma} \label{lem:swap}
Let $\G$ be a cubic graph of girth $6$
and let $G$ be a group of automorphisms of $\G$
acting transitively on its vertices.
Let $u_0$ be a vertex of $\G$.
Then there is a neighbour $u$ of $u_0$
such that there is an automorphism $\varphi \in G$ swapping $u_0$ and $u$.
If $u_1$ is such a neighbour, then,
assuming the configuration of Figure~\ref{fig:g6}(a),
we also have $v_{ij}^\varphi = v_{i'j'}$ and $w_{hij}^\varphi = w_{h'i'j'}$,
where $h \ne h'$, $i \ne i'$, $j \ne j'$.
\end{lemma}

\begin{proof}
Suppose that $\G$ has $n$ vertices.
Then it must have $3n$ arcs.
Clearly, the set of arcs of $\G$
is partitioned into at most three orbits under the action of $G$.
Suppose that there is an arc $s_1$ with tail $u_0$
such that $s_1$ and $s_1^{-1}$
(the inverse arc of $s_1$, see Section~\ref{ssec:graphs})
lie in distinct orbits.
By vertex-transitivity,
there is an arc $s_2$ with tail $u_0$
such that $s_2$ and $s_2^{-1}$ lie in the same orbits as $s_1^{-1}$ and $s_1$,
respectively.
Let $u$ be the head of the remaining arc $s_3$ with tail $u_0$.
Clearly, the arc $s_3$ cannot lie in the same orbits as $s_1$ or $s_2$,
so it must lie in its own orbit which then also contains $s_3^{-1}$.
Therefore, there is an automorphism $\varphi$ swapping $u_0$ and $u$.
If $u_1$ is a vertex like $u$, then,
without loss of generality,
$\varphi$ acts on the vertices $v_{ij}$ and $w_{hij}$
from Figure~\ref{fig:g6}(a) as described.
\end{proof}

\begin{lemma} \label{lem:g6c7}
Let $\G$ be a cubic girth-regular graph of girth $6$
with signature $(a, b, c)$.
Then $c \ne 7$.
\end{lemma}

\begin{proof}
Assuming the configuration of Figure~\ref{fig:g6}(a),
suppose that the arc $u_0 u_1$ lies on precisely seven $6$-cycles.
Without loss of generality,
we may assume that $w_{011}$ and $w_{100}$
only have one neighbour among $w_{hij}$ $(h, i, j \in \{0, 1\})$.
Clearly, these vertices must induce a $7$-path,
say $w_{011} w_{111} w_{001} w_{101} w_{010} w_{110} w_{000} w_{100}$,
see Figure~\ref{fig:g6}(b).
The arc $u_0 v_{00}$ thus lies on seven $6$-cycles,
and the arc $u_0 v_{01}$ lies on six $6$-cycles,
so the signature of $\G$ is $(6, 7, 7)$.
This implies that the arc $v_{01} w_{010}$ should lie on seven $6$-cycles,
however, it only lies on six $6$-cycles -- contradiction.
\end{proof}

\begin{lemma} \label{lem:g6c6}
Let $\G$ be a connected cubic vertex-transitive graph of girth $6$
with signature $(a, b, c)$, where $c = 6$.
Then $a = b = 6$ and $\G$ is the Möbius-Kantor graph,
which is isomorphic to $\Psi_8$.
\end{lemma}

\begin{proof}
Assuming the configuration of Figure~\ref{fig:g6}(a),
suppose that the arc $u_0 u_1$ lies on precisely six $6$-cycles.
First, assume that $T(u_0 u_1) = \{\{2, 2\}, \{2, 0\}\}$.
Without loss of generality,
we may assume that $w_{011}$
has no neighbours among $w_{hij}$ $(h, i, j \in \{0, 1\})$.
The remaining vertices must then induce a $6$-path,
say $w_{111} w_{001} w_{101}$ $w_{010} w_{110} w_{000} w_{100}$,
see Figure~\ref{fig:g6}(c).
Thus, we have $T(u_1 u_0) = \{\{2, 1\}, \{2, 1\}\}$.
The arc $u_0 v_{00}$ has the same asymmetric type as $u_0 u_1$,
contradiction.

Now, assume that $T(u_0 u_1) = \{\{2, 2\}, \{1, 1\}\}$.
Without loss of generality we may assume
$w_{000} \sim w_{100}, w_{110}$, $w_{001} \sim w_{101}, w_{111}$
and $w_{011} \sim w_{100}$,
see Figure~\ref{fig:g6c6}(a).
The vertex $w_{010}$ must then be
adjacent to one of $w_{101}$, $w_{110}$ and $w_{111}$,
giving $T(u_0 v_{00}) = T(u_0 u_1)$.
The arc $u_0 v_{01}$ then lies on precisely four girth cycles,
so the types of the arcs $u_0 u_1$ and $u_0 v_{00}$
must then be both symmetric.
However, in all three cases at least one of $T(u_1 u_0)$ and $T(v_{00} u_0)$
equals $\{\{2, 1\}, \{2, 1\}\}$,
making this case impossible.

\begin{figure}[t]
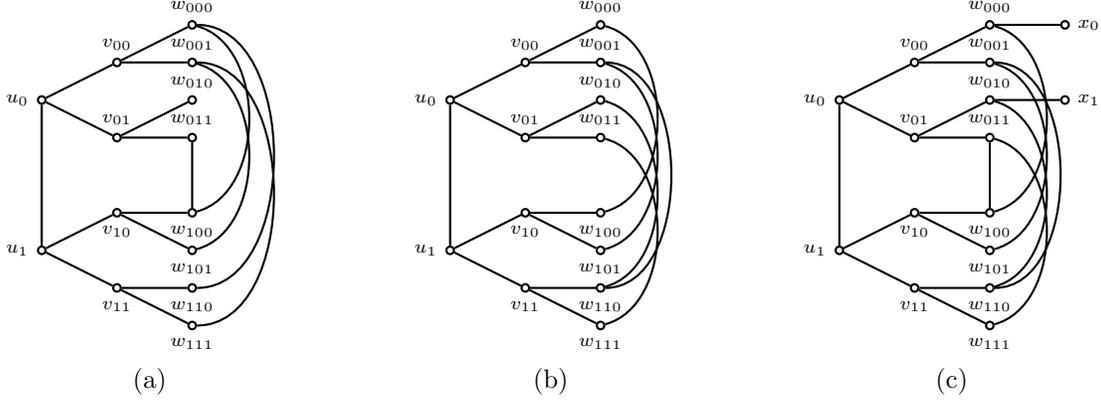

\makebox[\textwidth][c]{
\begin{tabular}{C{0.3}C{0.3}C{0.3}}
\leavevmode%
\beginpgfgraphicnamed{fig-g6-c6b}
\input{tikz/g6-c6b.tikz}
\endpgfgraphicnamed
&
\leavevmode%
\beginpgfgraphicnamed{fig-g6-c6c}
\input{tikz/g6-c6c.tikz}
\endpgfgraphicnamed
&
\leavevmode%
\beginpgfgraphicnamed{fig-g6-c6d}
\input{tikz/g6-c6d.tikz}
\endpgfgraphicnamed
\\
(a) & (b) & (c)
\end{tabular}
}
\caption{Constructing a graph of girth $6$ with $c = 6$.
(a) The case $T(u_0 u_1) = \{\{2, 2\}, \{1, 1\}\}$,
which leads to a contradiction.
(b) The case $T(u_0 u_1) = \{\{2, 1\}, \{2, 1\}\}$.
(c) The subcase of (b) with $a = b = 5$,
which also leads to a contradiction.
}
\label{fig:g6c6}
\end{figure}

The only remaining possibility is
$T(u_0 u_1) = T(u_1 u_0) = \{\{2, 1\}, \{2, 1\}\}$.
Without loss of generality,
we may assume $w_{000} \sim w_{100}$ and $w_{001} \sim w_{101}, w_{110}$.
Since $w_{111}$ cannot have $2$ neighbours among $w_{0ij}$
($i, j \in \{0, 1\}$),
we may further assume $w_{010} \sim w_{110}$ and $w_{011} \sim w_{111}$,
see Figure~\ref{fig:g6c6}(b).
There is another arc $w_{01i} w_{10j}$ for some $i, j \in \{0, 1\}$.
Note that $(i, j) = (0, 1)$ is not possible as that would give a $4$-cycle.

The arcs $u_0 v_{00}$ and $u_0 v_{01}$ lie on at least five $6$-cycles.
Suppose that $u_0 v_{00}$ lies on precisely five $6$-cycles.
Then $u_0 v_{01}$ must also lie on precisely five $6$-cycles,
and there is an automorphism of $\G$ swapping $u_0$ and $u_1$
as in Lemma~\ref{lem:swap}.
Thus, we have $w_{011} \sim w_{100}$.
Let $x_0$ and $x_1$ be the remaining neighbours of $w_{001}$ and $w_{110}$,
respectively (see Figure~\ref{fig:g6c6}(c))
-- since $u_0 v_{00}$ and $u_0 v_{01}$ lie on precisely five $6$-cycles,
they must be distinct vertices.
The arcs $u_0 v_{00}$ and $u_0 v_{01}$ now both have asymmetric type
$(\{\{2, 1\}, \{1, 1\}\}, \{\{2, 1\}, \{2, 0\}\})$, contradiction.

Therefore, $u_0 v_{00}$ lies on six $6$-cycles.
Let $x_0$ be the remaining neighbour of $w_{000}$.
Then one of $w_{010}$ and $w_{011}$ is adjacent to $x_0$,
while the other is adjacent to one of $w_{100}$ and $w_{101}$.
The arc $u_0 v_{01}$ then also lies on $6$ girth cycles,
so the signature of $\G$ is $(6, 6, 6)$.
Since we haven't made any other assumptions about the arc $u_0 u_1$,
we may then assume without loss of generality
that there is an automorphism of $\G$ swapping $u_0$ and $u_1$
as in Lemma~\ref{lem:swap}.
We thus have $w_{010} \sim x_0$, $w_{011} \sim w_{100}$,
and a vertex $x_1$ such that $w_{101}, w_{111} \sim x_1$.
For the arc $v_{00} w_{001}$ to lie on $6$ vertices,
we must also have $x_0 \sim x_1$,
which completes the graph.

Figure~\ref{fig:moebius-kantor}
shows the labelling of vertices of $\G$ with elements of $\D_8$,
establishing that $\G$ is isomorphic to $\Psi_8$,
and a drawing of $\G$ as a generalized Petersen graph $\GP(8, 3)$,
showing that it is also isomorphic to the Möbius-Kantor graph.
\end{proof}

\begin{figure}[t]
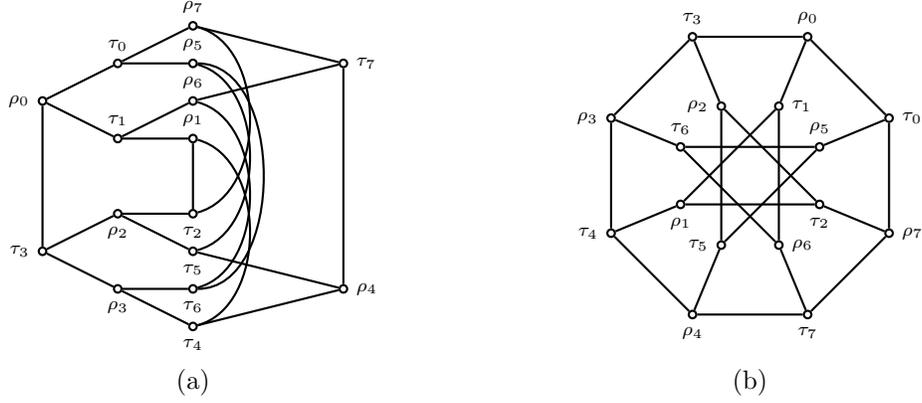

\makebox[\textwidth][c]{
\begin{tabular}{C{0.45}C{0.4}}
\leavevmode%
\beginpgfgraphicnamed{fig-g6-c6e}
\input{tikz/g6-c6e.tikz}
\endpgfgraphicnamed
&
\leavevmode%
\beginpgfgraphicnamed{fig-moebius-kantor}
\input{tikz/moebius-kantor.tikz}
\endpgfgraphicnamed
\\
(a) & (b)
\end{tabular}
}
\caption{The Möbius-Kantor graph of girth $6$ and signature $(6, 6, 6)$,
labelled as the Cayley graph $\Psi_8$,
(a) as the completion of Figure~\ref{fig:g6c6}(b),
and (b) as a generalized Petersen graph $\GP(8, 3)$.}
\label{fig:moebius-kantor}
\end{figure}

\begin{lemma} \label{lem:g6bc5}
Let $\G$ be a connected cubic vertex-transitive graph of girth $6$
with signature $(a, b, c)$, where $b = c = 5$.
Then $a = 4$ and $\G$ is isomorphic to $\Psi_9$.
\end{lemma}

\begin{proof}
As $a+b+c$ is even, $a$ must also be even.
Assuming the configuration of Figure~\ref{fig:g6}(a),
suppose that the arc $u_0 u_1$ lies on precisely $a$ $6$-cycles.
As $a \ne b, c$,
there exists an automorphism of $\G$ swapping $u_0$ and $u_1$
as in Lemma~\ref{lem:swap}.
By Lemma~\ref{lem:azero}, we must have $a \ge 2$.

Suppose $a = 2$.
Without loss of generality,
we may assume that the $2$-paths
$u_1 u_0 v_{00}$, $u_1 u_0 v_{01}$ and $v_{00} u_0 v_{01}$
lie on one, one and four $6$-cycles, respectively.
By symmetry, we may assume, say,
$w_{000} \sim w_{111}$ and $w_{011} \sim w_{100}$,
see Figure~\ref{fig:g6bc5}(a).
For $v_{00} u_0 v_{01}$ to lie on four $6$-cycles,
the vertices $w_{00i}$ and $w_{01j}$ should have a common neighbour
for all choices of $i, j \in \{0, 1\}$.
This is, however, not attainable.

\begin{figure}[t]
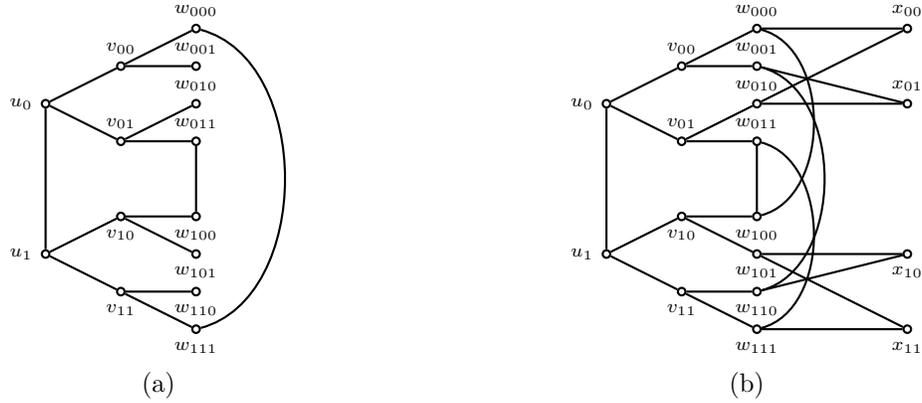

\makebox[\textwidth][c]{
\begin{tabular}{C{0.4}C{0.5}}
\leavevmode%
\beginpgfgraphicnamed{fig-g6-bc5a}
\input{tikz/g6-bc5a.tikz}
\endpgfgraphicnamed
&
\leavevmode%
\beginpgfgraphicnamed{fig-g6-bc5b}
\input{tikz/g6-bc5b.tikz}
\endpgfgraphicnamed
\\
(a) & (b)
\end{tabular}
}
\caption{Constructing a graph of girth $6$ with $b = c = 5$.
(a) The case $a = 2$, which cannot be completed.
(b) The case $a = 4$, which can be completed to $\Psi_9$.}
\label{fig:g6bc5}
\end{figure}

Therefore, we have $a = 4$.
We may thus assume that the $2$-paths
$u_1 u_0 v_{00}$, $u_1 u_0 v_{01}$ and $v_{00} u_0 v_{01}$
lie on two, two and three $6$-cycles, respectively.
Three pairs of vertices $w_{0ij}$ and $w_{0hk}$
for some $i, j, h, k \in \{0, 1\}$ with $i \ne h$
then have a common neighbour
-- this covers six of the remaining arcs
with tail among $w_{0ij}$ ($i, j \in \{0, 1\}$).
As there are eight such remaining arcs,
four of which have the head among $w_{1hk}$ ($h, k \in \{0, 1\}$),
it follows that one or two of the common neighbours
are vertices $w_{1hk}$ for some $h, k \in \{0, 1\}$.
Without loss of generality and by symmetry,
we may then assume that there is a $3$-path $w_{000} w_{100} w_{011} w_{111}$.
Now, $w_{011}$ cannot have a common neighbour with $w_{001}$
(as the common neighbour should be $w_{111}$,
and symmetry would imply that the $2$-path $u_1 u_0 v_{00}$
lied on three $6$-cycles),
so $w_{010}$ has common neighbours with both $w_{000}$ and $w_{001}$,
none of which can be among $w_{1hk}$ ($h, k \in \{0, 1\}$).
We may thus add new vertices $x_{ij}$ ($i, j \in \{0, 1\}$)
with $w_{iii} \sim x_{ii}$, $w_{iij} \sim x_{ij}$
and $w_{iji} \sim x_{ii}, x_{ij}$
for both choices of $\{i, j\} = \{0, 1\}$.
We also have $w_{001} \sim w_{110}$,
see Figure~\ref{fig:g6bc5}(b).

The arc $u_0 u_1$ has symmetric type
with $T(u_0 u_1) = \{\{2, 0\}, \{1, 1\}\}$.
Consider the arc $v_{00} w_{000}$.
We have $P(v_{00} w_{000}) = \{\{u_1, v_{01}\}, \{w_{101}, x_{01}\}\}$
and $P(w_{000} v_{00}) = \{\{v_{10}, w_{011}\}, \{w_{010}, y\}\}$,
where $y$ is the remaining neighbour of $x_{00}$.
The vertex $v_{01}$ is adjacent to $w_{010}$ and $w_{011}$,
while $u_1$ is adjacent to $v_{10}$,
meaning that $T(v_{00} w_{000}) \ne T(u_0 u_1)$.
The arc $v_{00} w_{000}$ thus lies on five $6$-cycles.
As $u_0 v_{00}$ also lies on five $6$-cycles,
$v_{00} w_{001}$ must lie on precisely four $6$-cycles.
We have $P(v_{00} w_{001}) = \{\{u_1, v_{01}\}, \{w_{100}, x_{00}\}\}$
and $P(w_{001} v_{00}) = \{\{v_{11}, x_{10}\}, \{w_{010}, z\}\}$,
where $z$ is the remaining neighbour of $x_{01}$.
As $w_{100}$ is not adjacent to any vertex of $P(w_{001} v_{00})$
and $x_{00} \sim w_{010}$,
we must also have $x_{00} \sim x_{10}$,
and, by symmetry, $x_{01} \sim x_{11}$,
thus completing the graph.
Figure~\ref{fig:g6c5}(a) shows the labelling of vertices of $\G$
with elements of $\D_9$,
establishing that $\G$ is isomorphic to $\Psi_9$.
\end{proof}

\begin{figure}[t]
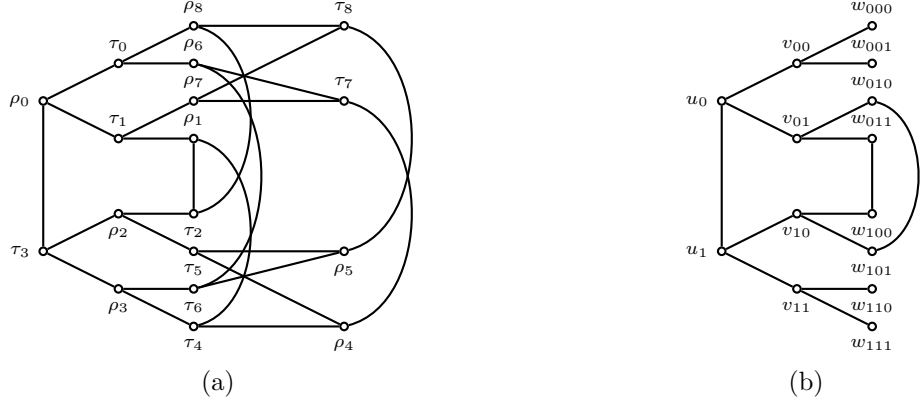

\makebox[\textwidth][c]{
\begin{tabular}{C{0.5}C{0.4}}
\leavevmode%
\beginpgfgraphicnamed{fig-g6-bc5c}
\input{tikz/g6-bc5c.tikz}
\endpgfgraphicnamed
&
\leavevmode%
\beginpgfgraphicnamed{fig-g6-c5a}
\input{tikz/g6-c5a.tikz}
\endpgfgraphicnamed
\\
(a) & (b)
\end{tabular}
}
\caption{Constructing a graph of girth $6$ with $c = 5$.
(a) The graph $\Psi_9$ with signature $(4, 5, 5)$.
(b) The case $a = 2$, $b = 3$, which cannot be completed.
}
\label{fig:g6c5}
\end{figure}

\begin{lemma} \label{lem:g6c5}
Let $\G$ be a connected cubic vertex-transitive graph of girth $6$
with signature $(a, b, c)$, where $b < c = 5$.
Then $a = 3$, $b = 4$
and $\G$ is isomorphic to $\Psi_n$ for some $n \ge 10$.
\end{lemma}

\begin{proof}
As $a+b+c$ is even, $a+b$ must be odd.
Assuming the configuration of Figure~\ref{fig:g6}(a),
suppose that the arc $u_0 u_1$ lies on precisely $a$ $6$-cycles.
As $a \ne b, c$,
there exists an automorphism $\varphi$ of $\G$ swapping $u_0$ and $u_1$
as in Lemma~\ref{lem:swap}.
By Lemma~\ref{lem:azero}, we must have $a \ge 1$.

Suppose $a = 1$.
By triangle inequality, we then have $b = 4$.
Without loss of generality,
we may assume that the $2$-paths
$u_1 u_0 v_{00}$, $u_1 u_0 v_{01}$ and $v_{00} u_0 v_{01}$
lie on zero, one and four $6$-cycles, respectively.
Then,
each remaining arc with tail among $w_{0ij}$ ($i, j \in \{0, 1\}$)
must have a common neighbour with another such vertex as its head.
However, the single edge completing a $6$-cycle on $u_1 u_0 v_{01}$
does not reach such a common neighbour, contradiction.

Now suppose $a = 2$.
We must then have $b = 3$.
Without loss of generality,
we may assume that the $2$-paths
$u_1 u_0 v_{00}$, $u_1 u_0 v_{01}$ and $v_{00} u_0 v_{01}$
lie on zero, two and three $6$-cycles, respectively.
By symmetry, we then have $w_{010} \sim w_{101}$ and $w_{011} \sim w_{100}$,
see Figure~\ref{fig:g6c5}(b).
For $v_{00} u_0 v_{01}$ to lie on three $6$-cycles,
the vertices $w_{00i}$ and $w_{01j}$ should have a common neighbour
for three choices of $i, j \in \{0, 1\}$.
This is, however, not attainable.

The only remaining option is $a = 3$, $b = 4$.
Without loss of generality,
we may assume that the $2$-paths
$u_1 u_0 v_{00}$, $u_1 u_0 v_{01}$ and $v_{00} u_0 v_{01}$
lie on one, two and three $6$-cycles, respectively.
Now, the vertices $w_{01h}$ ($h \in \{0, 1\}$)
have two neighbours among $w_{1ij}$ ($i, j \in \{0, 1\}$),
but at most one with $i = 1$.
Without loss of generality we may then assume $w_{011} \sim w_{100}$,
and also that $w_{000}$ has a neighbour among $w_{1ij}$ ($i, j \in \{0, 1\}$).
By symmetry,
$w_{111}$ then has a neighbour among $w_{0ij}$ ($i, j \in \{0, 1\}$).
Since only one of $v_{10}$ and $v_{11}$
has a common neighbour with only one of $w_{000}$ and $w_{001}$,
and the arc $u_0 v_{00}$ lies on four $6$-cycles,
the latter must have symmetric type
with $T(u_0 v_{00}) = \{\{2, 1\}, \{1, 0\}\}$.

In the remainder of this proof,
we will gradually build an isomorphism
between $\G$ and $\Psi_n$ for some $n \ge 10$.
We will thus assume that the graph $\Gamma$ has $2n$ vertices
which the automorphism maps to the elements of the group $\D_n$.
We start by relabelling the vertices
$u_0$, $u_1$, $v_{00}$, $v_{01}$, $v_{10}$, $v_{11}$, $w_{000}$, $w_{001}$,
$w_{010}$, $w_{011}$, $w_{100}$, $w_{101}$, $w_{110}$ and $w_{111}$
as $\rho_0$, $\tau_3$, $\tau_0$, $\tau_1$, $\rho_2$, $\rho_3$, $\rho_{-1}$,
$\rho_{-3}$, $\rho_{-2}$, $\rho_1$, $\tau_2$, $\tau_5$, $\tau_6$ and $\tau_4$,
respectively.
Note that each arc determined so far
is of form $\rho_i \tau_j$ or $\tau_j \rho_i$,
where $j-i \in \{0, 1, 3\}$.
Also, the automorphism $\varphi$ acts as
$\rho_i^\varphi = \tau_{3-i}$ and $\tau_i^\varphi = \rho_{3-i}$
on the vertices determined so far.
These properties will continue to hold
as we will be determining more arcs and vertices.

By the above argument, we may add vertices
$\tau_{-1}$, $\tau_{-3}$, $\tau_{-2}$, $\rho_5$, $\rho_6$ and $\rho_4$
such that $\rho_{-1} \sim \tau_{-1}$,
$\tau_{-2} \sim \rho_{-3} \sim \tau_{-3}$,
$\rho_5 \sim \tau_6 \sim \rho_6$ and $\tau_4 \sim \rho_4$.
Since none of $\tau_i$ ($i \in \{-1, -2, -3\}$)
can have two neighbours among $\rho_j$ ($j \in \{-2, 1, 2, 3\}$),
it follows that the remaining neighbour of $\rho_{-1}$
is either $\tau_2$ or $\tau_5$.
However, $\rho_{-1} \sim \tau_5$ implies $\rho_{-2} \sim \tau_4$ by symmetry
(see Figure~\ref{fig:g6s345}(a)),
and the vertices $\rho_j$ ($j \in \{-2, 1\}$)
cannot have three neighbours among $\tau_i$ ($i \in \{-1, -2, -3, 5\}$),
contradiction.

\begin{figure}[t]
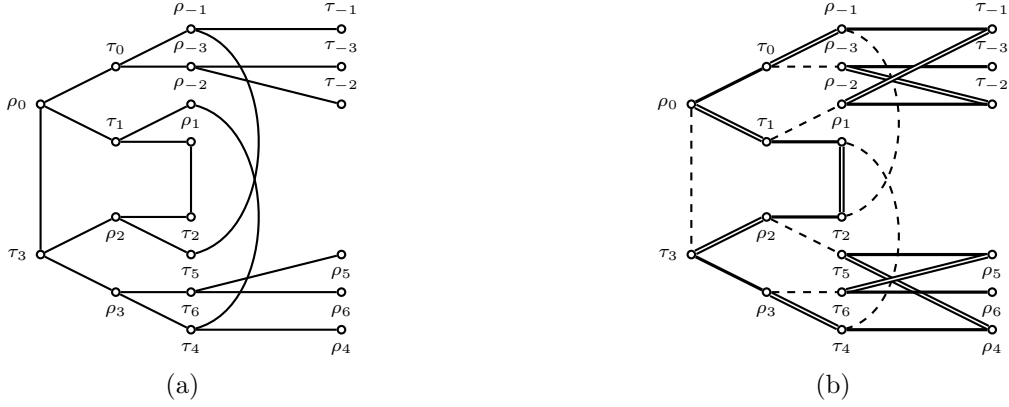

\makebox[\textwidth][c]{
\begin{tabular}{C{0.5}C{0.5}}
\leavevmode%
\beginpgfgraphicnamed{fig-g6-c5b}
\input{tikz/g6-c5b.tikz}
\endpgfgraphicnamed
&
\leavevmode%
\beginpgfgraphicnamed{fig-g6-c5c}
\input{tikz/g6-c5c.tikz}
\endpgfgraphicnamed
\\
(a) & (b)
\end{tabular}
}
\caption{Constructing a graph of girth $6$ with signature $(3, 4, 5)$,
with two choices of the remaining neighbour of $\rho_{-1}$.
In (a), it is assumed that $\rho_{-1} \sim \tau_5$,
which leads to a contradiction.
In (b), it is assumed that $\rho_{-1} \sim \tau_2$.
The dashed, thick and double edges
lie on three, four and five $6$-cycles, respectively.}
\label{fig:g6s345}
\end{figure}

Therefore, we have $\rho_{-1} \sim \tau_2$,
and by symmetry also $\rho_1 \sim \tau_4$.
Without loss of generality,
we now have $\tau_{-1} \sim \rho_{-2} \sim \tau_{-2}$,
and by symmetry also $\rho_4 \sim \tau_5 \sim \rho_5$,
see Figure~\ref{fig:g6s345}(b).
Examining the arcs $\rho_0 \tau_3$, $\rho_0 \tau_0$ and $\rho_0 \tau_1$,
it follows that an arc $s$ lying on three, four or five $6$-cycles
has $T(s)$ equal to $\{\{2, 0\}, \{1, 0\}\}$, $\{\{2, 1\}, \{1, 0\}\}$
and $\{\{2, 1\}, \{1, 1\}\}$, respectively.
As $\tau_0 \rho_0$ lies on four $6$-cycles
and both $\tau_1$ and $\tau_3$ have a common neighbour with $\tau_2$,
it follows that $\tau_0 \rho_{-1}$ lies on five $6$-cycles.
The edge $\tau_0 \rho_{-3}$ then lies on three $6$-cycles;
by symmetry,
$\rho_3 \tau_4$ and $\rho_3 \tau_6$ must lie on five and three $6$-cycles, respectively.
As $\tau_1 \rho_0$ lies on five $6$-cycles
and both $\tau_0$ and $\tau_3$ have a common neighbour with $\tau_2$,
it follows that $\tau_1 \rho_1$ lies on four $6$-cycles.
The edge $\tau_1 \rho_{-2}$ then lies on three $6$-cycles;
by symmetry,
$\rho_2 \tau_2$ and $\rho_2 \tau_5$ must lie on four and three $6$-cycles, respectively.
It follows that $\rho_{-1} \tau_2$ and $\rho_1 \tau_4$ lie on three $6$-cycles;
furthermore, $\rho_1 \tau_2$ must lie on five $6$-cycles,
while $\rho_{-1} \tau_{-1}$ and $\tau_4 \rho_4$ lie on four $6$-cycles each.
Continuing the examination,
we obtain that the paths $\tau_{-1} \rho_{-2} \tau_{-2} \rho_{-3} \tau_{-3}$
and $\rho_4 \tau_5 \rho_5 \tau_6 \rho_6$ both consist of arcs
alternatingly lying on five and four $6$-cycles each.

Now we have arrived at a point
where we have determined $4t+8$ vertices of the graph for some $t \ge 3$,
of which $\tau_{-t}$ and $\rho_{3+t}$
are missing two arcs lying on three and five $6$-cycles,
$\tau_{1-t}$, $\tau_{2-t}$, $\rho_{1+t}$ and $\rho_{2+t}$
are missing an arc lying on three $6$-cycles,
and the vertices $\rho_i$ and $\tau_{3+i}$ ($-t \le i \le t$)
have their neighbourhoods completely determined,
and the arcs $\rho_i \tau_j$ and $\tau_j \rho_i$ determined so far
lie on three, four or five $6$-cycles
precisely when $j-i$ equals $3$, $0$ and $1$, respectively.

For the arc $\tau_{3-t} \rho_{2-t}$
to have the desired type $\{\{2, 1\}, \{1, 1\}\}$,
the vertices $\tau_{-t}$ and $\tau_{2-t}$ must have a common neighbour.
By symmetry,
the vertices $\rho_{3+t}$ and $\rho_{1+t}$ must also have a common neighbour.
Of the vertices determined so far,
only $\rho_{3+t}$ is a candidate
for the common neighbour of $\tau_{-t}$ and $\tau_{2-t}$.
Suppose that this is the case
-- by symmetry, we must also have $\tau_{-t} \sim \rho_{1+t}$.
As the arc $\tau_{-t} \rho_{3+t}$ must lie on five $6$-cycles
and each of the vertices $\tau_{3-t}$, $\tau_{1+t}$ and $\tau_{2+t}$
has a common neighbour with precisely one of $\tau_{2-t}$ and $\tau_{3+t}$,
it follows that $\tau_{1-t} \sim \rho_{2+t}$,
which completes the graph, see for example Figure~\ref{fig:psieven}(a).
If the indices are taken modulo $n = 2t+4$,
it can be seen that the graph $\G$ is isomorphic to $\Psi_n$.

\begin{figure}[t]
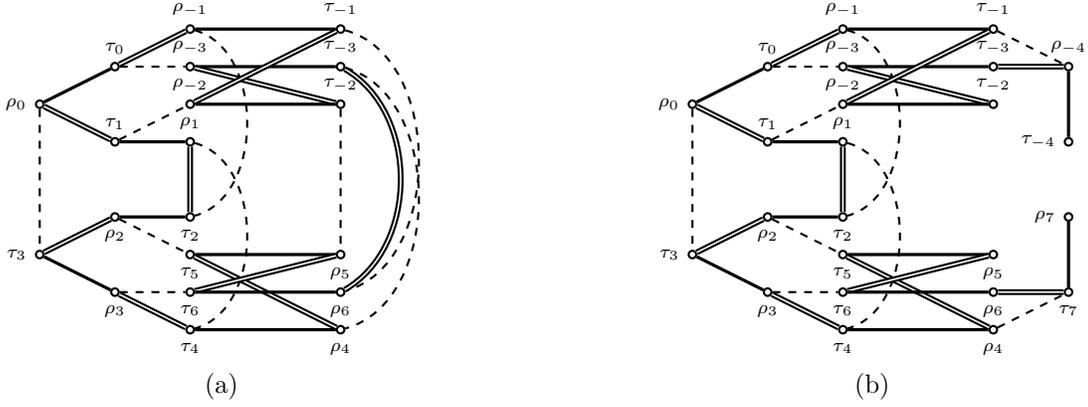

\makebox[\textwidth][c]{
\begin{tabular}{C{0.5}C{0.5}}
\leavevmode%
\beginpgfgraphicnamed{fig-g6-c5d}
\input{tikz/g6-c5d.tikz}
\endpgfgraphicnamed
&
\leavevmode%
\beginpgfgraphicnamed{fig-g6-c5e}
\input{tikz/g6-c5e.tikz}
\endpgfgraphicnamed
\\
(a) & (b)
\end{tabular}
}
\caption{Constructing a graph of girth $6$ with signature $(3, 4, 5)$,
with two choices of the common neighbour of $\tau_{-1}$ and $\tau_{-3}$.
In (a), it is assumed that the common neighbour is $\rho_6$,
which then gives the graph $\Psi_{10}$.
In (b), it is assumed that the common neighbour is a new vertex $\rho_{-4}$,
which is adjacent to another new vertex $\tau_{-4}$,
thus giving a situation similar to Figure~\ref{fig:g6s345}(b).
The dashed, thick and double edges
lie on three, four and five $6$-cycles, respectively.}
\label{fig:psieven}
\end{figure}

Now assume the contrary,
i.e., that the common neighbour of $\tau_{2-t}$ and $\tau_{-t}$
is not $\rho_{3+t}$.
Then it must be a new vertex,
which we name $\rho_{-1-t}$.
By symmetry, the common neighbour of $\rho_{1+t}$ and $\rho_{3+t}$
must be another new vertex -- call it $\tau_{4+t}$.
The arcs $\tau_{2-t} \rho_{-1-t}$ and $\rho_{1+t} \tau_{4+t}$
must lie on three $6$-cycles each,
and the arcs $\tau_{-t} \rho_{-1-t}$ and $\rho_{3+t} \tau_{4+t}$
must then lie on five $6$-cycles each.
Therefore,
the remaining arcs with $\rho_{-1-t}$ and $\tau_{4+t}$ as tails
must lie on four $6$-cycles each.
If the remaining neighbours of $\rho_{-1-t}$ and $\tau_{4+t}$
are new vertices $\tau_{-1-t}$ and $\rho_{4+t}$
(see for example Figure~\ref{fig:psieven}(b)),
then we are back at the previous case.

If, on the other hand,
$\rho_{-1-t}$ and $\tau_{4+t}$ are adjacent to known vertices,
we must have $\rho_{-1-t} \sim \tau_{4+t}$,
since all other vertices missing an arc
already lie on an arc lying on four $6$-cycles.
As the arc $\tau_{2-t} \rho_{-1-t}$ lies on precisely three $6$-cycles
and the vertex $\rho_{1+t}$ is not adjacent
to any of $\tau_{i-t}$ ($i \in \{1, 3, 4, 5\}$),
we must have $\rho_{3+t} \sim \tau_{1-t}$,
and by symmetry also $\tau_{-t} \sim \rho_{2+t}$,
again completing the graph
(see for example Figure~\ref{fig:psiodd}(a)).
If the indices are taken modulo $n = 2t+5$,
it can be seen that the graph $\G$ is isomorphic to $\Psi_n$.
\end{proof}

\begin{figure}[t]
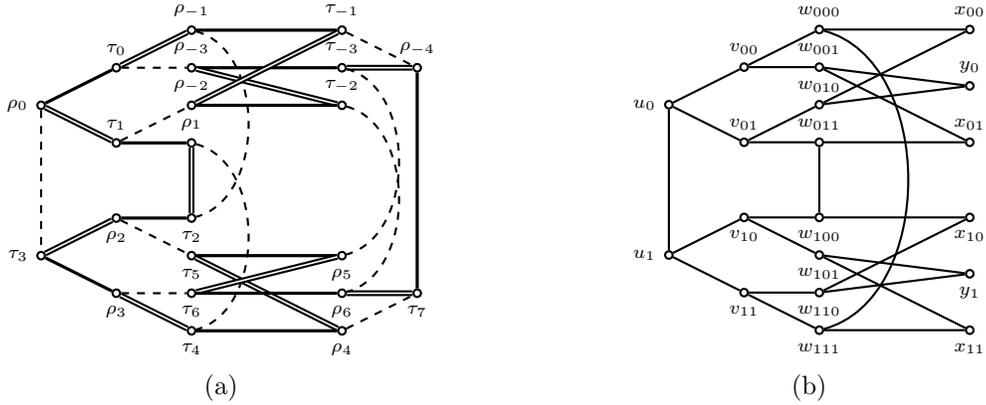

\makebox[\textwidth][c]{
\begin{tabular}{C{0.5}C{0.4}}
\leavevmode%
\beginpgfgraphicnamed{fig-g6-c5f}
\input{tikz/g6-c5f.tikz}
\endpgfgraphicnamed
&
\leavevmode%
\beginpgfgraphicnamed{fig-g6-c4a}
\input{tikz/g6-c4a.tikz}
\endpgfgraphicnamed
\\
(a) & (b)
\end{tabular}
}
\caption{(a) Completing Figure~\ref{fig:psieven}(b)
by identifying $\rho_{-4} = \rho_7$ and $\tau_{-4} = \tau_7$
to obtain $\Psi_{11}$.
The dashed, thick and double edges
lie on three, four and five $6$-cycles, respectively.
(b) Constructing a graph of girth $6$ with signature $(2, 4, 4)$,
which leads to a contradiction.}
\label{fig:psiodd}
\end{figure}

\begin{lemma} \label{lem:g6c4}
Let $\G$ be a connected cubic vertex-transitive graph of girth $6$
with signature $(a, b, c)$, where $c = 4$.
Then $a = b = 4$ and either $\G$ is the Pappus graph,
which is isomorphic to $\Sigma_3$,
or $\G$ is the Desargues graph.
\end{lemma}

\begin{proof}
First assume that $a < b = 4$.
By Lemma~\ref{lem:azero}, we must then have $a = 2$.
Assuming the configuration of Figure~\ref{fig:g6}(a),
suppose that the arc $u_0 u_1$ lies on precisely two $6$-cycles.
As $a \ne b, c$,
there exists an automorphism of $\G$ swapping $u_0$ and $u_1$
as in Lemma~\ref{lem:swap}.
The $2$-path $v_{00} u_0 v_{01}$ then lies on three $6$-cycles,
so the vertices $w_{00i}$ and $w_{01j}$ have a common neighbour
for three choices of $i, j \in \{0, 1\}$.
Symmetry implies that none of these common neighbours
is a vertex $w_{1h\ell}$ ($h, \ell \in \{0, 1\}$).
Instead we may assume without loss of generality
that $w_{000} \sim w_{111}$ and $w_{011} \sim w_{100}$.
The vertices $w_{000}$ and $w_{011}$ cannot have a common neighbour,
as that would imply that there are
two common neighbours of $w_{001}$ and $w_{010}$,
giving us a quadrangle.
By symmetry, $w_{100}$ and $w_{111}$ also have no common neighbour.
Therefore, we have new vertices
$x_{00}$, $y_0$, $x_{01}$, $x_{10}$, $y_1$ and $x_{11}$
such that $w_{000} x_{00} w_{010} y_0 w_{001} x_{01} w_{011}$
and $w_{111} x_{11} w_{101} y_1 w_{110} x_{10} w_{100}$
are $6$-paths in $\G$,
see Figure~\ref{fig:psiodd}(b).
However, the arcs $u_0 v_{00}$ and $u_0 v_{01}$
now both have the same asymmetric type
$(\{\{2, 1\}, \{1, 0\}\}, \{\{1, 1\}, \{1, 1\}\})$, contradiction.

Therefore, we have either $b < 4$ or $a = b = 4$.
In either case,
there exists an arc $u_0 u_1$ lying on four $6$-cycles,
where we again assume the configuration of Figure~\ref{fig:g6}(a),
and an automorphism $\varphi$ of $\G$ swapping $u_0$ and $u_1$
as in Lemma~\ref{lem:swap}.
Clearly, $u_0 u_1$ must have symmetric type.
It cannot have type $T(u_0 u_1) = \{\{2, 2\}, \{0, 0\}\}$
or $T(u_0 u_1) = \{\{2, 0\}, \{2, 0\}\}$,
as this would imply the existence of a quadrangle.
Suppose that $T(u_0 u_1) = \{\{2, 1\}, \{1, 0\}\}$.
Without loss of generality,
we may then assume $w_{000} \sim w_{100}$, $w_{011} \sim w_{111}$
(see Figure~\ref{fig:g6c4t2}(a)),
and $w_{010} \sim w_{10i}$, $w_{011} \sim w_{10j}$
for some choice of $\{i, j\} = \{0, 1\}$.
Now, $v_{10}$ has precisely one neighbour
among the vertices of $P(v_{00} u_0)$,
$v_{11}$ has none,
and $w_{010}$ and $w_{011}$ together have one or two.
As the arc $u_0 v_{01}$ lies on four $6$-cycles,
it follows that $u_0 v_{00}$ must then lie on precisely two $6$-cycles.
This is however not possible by the previous argument.

\begin{figure}[t]
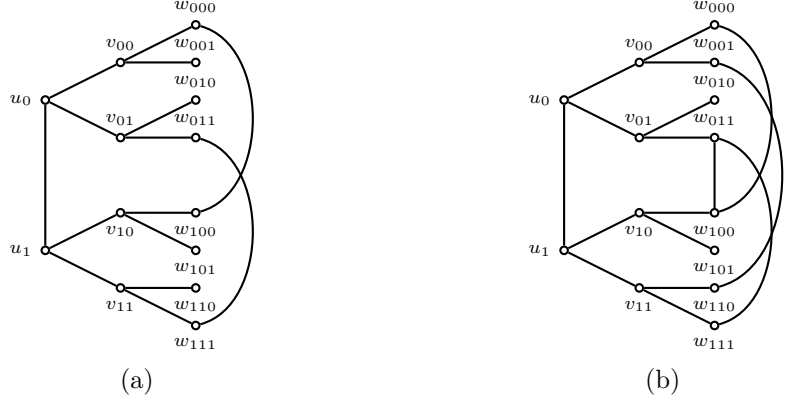

\makebox[\textwidth][c]{
\begin{tabular}{C{0.4}C{0.4}}
\leavevmode%
\beginpgfgraphicnamed{fig-g6-c4b}
\input{tikz/g6-c4b.tikz}
\endpgfgraphicnamed
&
\leavevmode%
\beginpgfgraphicnamed{fig-g6-c4c}
\input{tikz/g6-c4c.tikz}
\endpgfgraphicnamed
\\
(a) & (b)
\end{tabular}
}
\caption{Constructing a graph of girth $6$ with $c = 4$
with two choices for $T(u_0 u_1)$.
In (a), it is assumed that $T(u_0 u_1) = \{\{2, 1\}, \{1, 0\}\}$,
while in (b), it is assumed that $T(u_0 u_1) = \{\{2, 0\}, \{1, 1\}\}$.
Both assumptions lead to a contradiction.}
\label{fig:g6c4t2}
\end{figure}

Suppose now that $T(u_0 u_1) = \{\{2, 0\}, \{1, 1\}\}$.
Without loss of generality we may now assume
$w_{000} \sim w_{100} \sim w_{011} \sim w_{111}$ and $w_{001} \sim w_{110}$,
see Figure~\ref{fig:g6c4t2}(b).
Now, $v_{10}$, $v_{11}$ and $w_{011}$ each have one neighbour
among the vertices of $P(v_{00} u_0)$,
and $w_{010}$ can have at most one.
As $w_{100}$ is adjacent to both $v_{10}$ and $w_{011}$,
it follows that the arc $u_0 v_{00}$ has asymmetric type
with $T(u_0 v_{00}) = T(v_{01} u_0) = \{\{1, 1\}, \{1, r\}\}$
for some $r \in \{0, 1\}$.
However, $w_{100}$ is adjacent to $v_{10}$ and $w_{000}$,
contradicting such a type for the arc $v_{01} u_0$.

We must then conclude that $T(u_0 u_1) = \{\{1, 1\}, \{1, 1\}\}$.
Without loss of generality,
we may assume $w_{000} \sim w_{111}$ and $w_{011} \sim w_{100}$.
The remaining two $6$-cycles on $u_0 u_1$
can now be completed in two ways.
First assume $w_{001} \sim w_{110}$ and $w_{010} \sim w_{101}$,
see Figure~\ref{fig:g6c4t1}(a).
As no vertex can be adjacent
to three of the vertices $w_{0ij}$ ($i, j \in \{0, 1\}$),
it follows that the arcs $u_0 v_{00}$ and $u_0 v_{01}$
have the same asymmetric type
$(\{\{2, 0\}, \{x, y\}\}, \{\{1, r\}, \{1, s\}\})$
for some $r, s \in \{0, 1\}$, contradiction.

\begin{figure}[t]
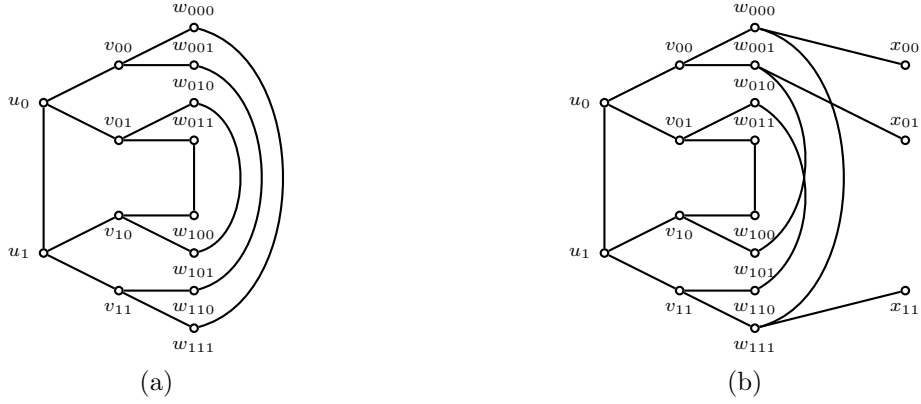

\makebox[\textwidth][c]{
\begin{tabular}{C{0.4}C{0.5}}
\leavevmode%
\beginpgfgraphicnamed{fig-g6-c4d}
\input{tikz/g6-c4d.tikz}
\endpgfgraphicnamed
&
\leavevmode%
\beginpgfgraphicnamed{fig-g6-c4e}
\input{tikz/g6-c4e.tikz}
\endpgfgraphicnamed
\\
(a) & (b)
\end{tabular}
}
\caption{Constructing a graph of girth $6$ with $c = 4$
and $T(u_0 u_1) = \{\{1, 1\}, \{1, 1\}\}$,
with two possibilities for completing the $6$-cycles on $u_0 u_1$.
In (a), $w_{001} \sim w_{110}$ and $w_{010} \sim w_{101}$ is assumed,
which leads to a contradiction.
In (b), $w_{001} \sim w_{101}$ and $w_{010} \sim w_{110}$ is assumed.}
\label{fig:g6c4t1}
\end{figure}

Therefore, we have $w_{001} \sim w_{101}$ and $w_{010} \sim w_{110}$.
Now we have $T(u_0 v_{00}) = T(u_0 v_{01}) = \{\{1, 1\}, \{r, s\}\}$
and $T(v_{00} u_0) = T(v_{01} u_0) = \{\{1, r\}, \{1, s\}\}$
for some $r, s \in \{0, 1\}$.
The arcs $u_0 v_{00}$ and $u_0 v_{01}$ therefore have symmetric types,
so we must have $r = 1$.
Thus, the arcs $u_0 v_{00}$ and $u_0 v_{01}$,
and by symmetry also $u_1 v_{10}$ and $u_1 v_{11}$,
all lie on either three or four $6$-cycles,
i.e., $a = b \in \{3, 4\}$.
In particular, if $a = b = 3$,
then the $2$-paths $v_{00} u_0 v_{01}$ and $v_{11} u_1 v_{10}$
lie on precisely one $6$-cycle each.
Let $x_{00}$ and $x_{01}$
be the remaining negibours
of $w_{000}$ and $w_{001}$,
respectively.
Without loss of generality,
we may then assume that $x_{00}$
is adjacent to a vertex $w_{01i}$ for some $i \in \{0, 1\}$.
By symmetry, $w_{111}$ must also have a common neighbour
with a vertex $w_{10j}$ for some $j \in \{0, 1\}$.
Since neither $x_{00}$ nor $x_{01}$ can be adjacent with $w_{111}$
(as that would give a triangle or a pentagon, respectively),
the common neighbour must be a new vertex $x_{11}$,
see Figure~\ref{fig:g6c4t1}(b).

Let us first assume $w_{010} \sim x_{00}$.
By symmetry, we then also have $w_{101} \sim x_{11}$.
Depending on whether the arc $v_{00} w_{001}$
lies on three or four $6$-cycles,
$x_{01}$ must have a common neighbour
with one or both of $v_{01}$ and $x_{00}$.
If $w_{011} \sim x_{01}$ were true,
then the arcs $u_0 v_{00}$ and $u_0 v_{01}$
would lie on four $6$-cycles,
and by vertex-transitivity,
this would be true of all arcs of $\G$.
It follows that $x_{00}$ and $x_{01}$ must have a common neighbour $y_0$
regardless of this condition.
If $w_{110}$ were adjacent to $x_{01}$,
then, by symmetry, $x_{11}$ would have to be adjacent to $y_0$,
giving us a pentagon.
If, on the other hand, $w_{110}$ were adjacent to $y_0$,
then symmetry would imply $x_{01} \sim x_{11}$,
and the neighbourhoods of all vertices determined so far would be determined,
with the execption of the adjacent vertices $w_{011}$ and $w_{100}$,
which are missing an arc each.
Removing these two vertices from the graph $\G$
would yield a disconnected graph;
since this is not true, say, for the vertex $u_0$ and any of its neighbours,
this contradicts vertex-transitivity of $\G$.
Thus, the remaining neighbour of $w_{110}$ must be a new vertex $x_{10}$,
which, by symmetry, has a common neighbour $y_1$ with $x_{11}$,
see Figure~\ref{fig:g6c4w0}(a).

\begin{figure}[t]
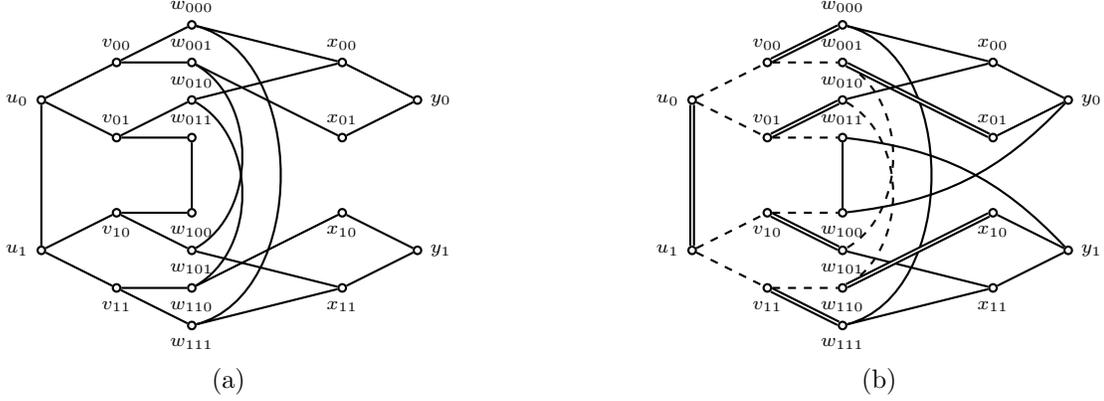

\makebox[\textwidth][c]{
\begin{tabular}{C{0.5}C{0.5}}
\leavevmode%
\beginpgfgraphicnamed{fig-g6-c4f}
\input{tikz/g6-c4f.tikz}
\endpgfgraphicnamed
&
\leavevmode%
\beginpgfgraphicnamed{fig-g6-c4g}
\input{tikz/g6-c4g.tikz}
\endpgfgraphicnamed
\\
(a) & (b)
\end{tabular}
}
\caption{(a) Constructing a graph of girth $6$ with $c = 4$
and $w_{010} \sim x_{00}$, $w_{101} \sim x_{11}$.
(b) Additionally assuming $a = b = 3$.
The dashed and double edges should lie on three and four $6$-cycles,
respectively, however,
this cannot be attained for $w_{001} x_{01}$ and $w_{110} x_{10}$.}
\label{fig:g6c4w0}
\end{figure}

Assume $a = b = 3$,
and let $z_0$ and $z_1$ be the remaining neighbours
of $w_{011}$ and $w_{100}$, respectively.
Then the vertex $v_{00}$ has no common neighbours with $w_{100}$ or $z_0$,
so the arcs $v_{01} w_{011}$ and,
by symmetry, also $v_{10} w_{100}$,
each lie on precisely three $6$-cycles.
As $w_{110}$ cannot have a common neighbour with $w_{100}$
(or $v_{11} u_1 v_{10}$ would lie on two $6$-cycles),
it must have a common neighbour with $z_0$,
which must be the vertex $x_{10}$.
Therefore, we have $y_1 = z_0$, and by symmetry also $y_0 = z_1$,
see Figure~\ref{fig:g6c4w0}(b).
The arcs $v_{01} w_{010}$ and $v_{10} w_{101}$
now lie on four $6$-cycles each.
It follows that $w_{001} w_{101}$ and $w_{010} w_{110}$
should lie on three $6$-cycles each.
As this is also true of $v_{00} w_{001}$ and $v_{11} w_{110}$,
the arcs $w_{001} x_{01}$ and $w_{110} x_{10}$
should lie on four $6$-cycles each,
which, however, cannot be attained.

We thus have $a = b = 4$ and $w_{011} \sim x_{01}$, $w_{100} \sim x_{10}$.
For the arc $w_{000} w_{111}$ to lie on four $6$-cycles,
we must have $y_0 \sim y_1$, which completes the graph.
Figure~\ref{fig:desargues} shows the graph $\G$
and the Desargues configuration
with points and lines labelled with the vertices of $\G$,
showing that $\G$ is indeed its incidence graph,
i.e., it is isomorphic to the Desargues graph.

\begin{figure}[t]
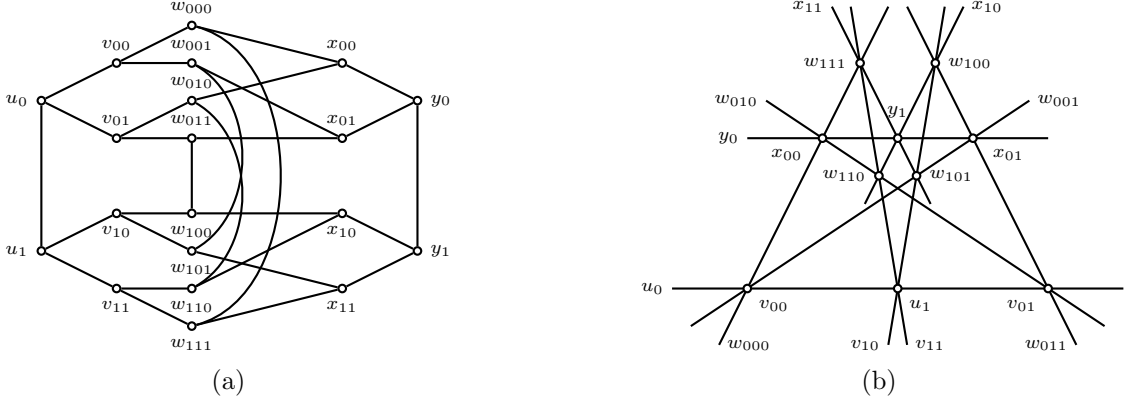

\makebox[\textwidth][c]{
\begin{tabular}{C{0.5}C{0.5}}
\leavevmode%
\beginpgfgraphicnamed{fig-g6-c4h}
\input{tikz/g6-c4h.tikz}
\endpgfgraphicnamed
&
\leavevmode%
\beginpgfgraphicnamed{fig-desargues}
\input{tikz/desargues.tikz}
\endpgfgraphicnamed
\\
(a) & (b)
\end{tabular}
}
\caption{(a) Completing Figure~\ref{fig:g6c4w0}(a)
to obtain the Desargues graph.
(b) The Desargues configuration
with points and lines labelled with the vertex labels of (a).}
\label{fig:desargues}
\end{figure}

Finally, we're left with the case when $w_{011} \sim x_{00}$.
By symmetry, we then also have $w_{100} \sim x_{11}$.
As $v_{01}$ would have no common neighbour with $x_{01}$
if the arc $v_{00} w_{001}$ lay on precisely three $6$-cycles
(or $v_{00} u_0 v_{01}$ would lie on two $6$-cycles),
each of the vertices $w_{111}$ and $x_{00}$
must have a common neighbour with one of $w_{101}$ and $x_{01}$
regardless of this condition.
As $w_{101}$ and $w_{111}$ have no common neighbour,
we must have $x_{01} \sim x_{11}$,
and the common neighbour of $w_{101}$ and $x_{00}$
must be a new vertex $x_{10}$.
By symmetry, we then also have
$w_{011} \sim x_{01}$ and $w_{110} \sim x_{10}$,
which completes the graph.
Figure~\ref{fig:pappus} shows the labelling of vertices of $\G$
with elements of $\Z_3 \times \D_3$,
establishing that $\G$ is isomorphic to $\Sigma_3$,
and the Pappus configuration
with points and lines labelled with the vertices of $\G$,
showing that $\G$ is indeed its incidence graph,
i.e., it is isomorphic to the Pappus graph.
\end{proof}

\begin{figure}[t]
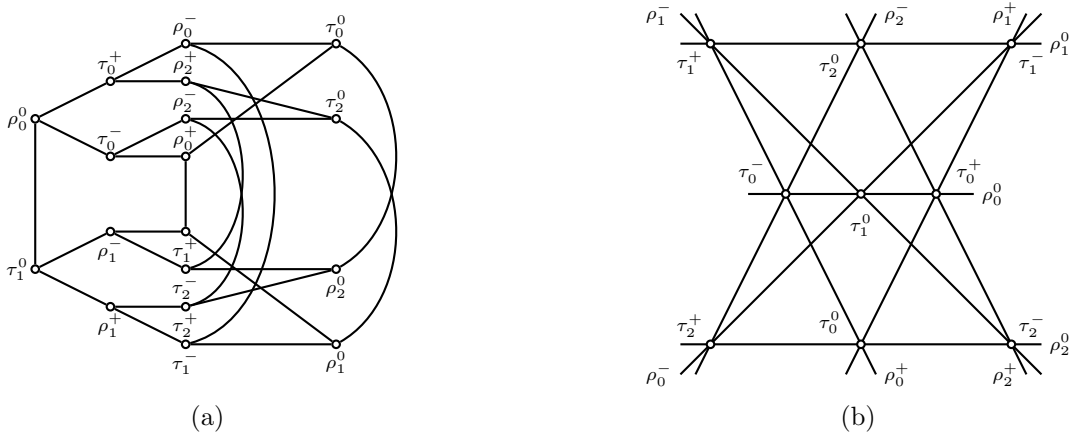

\makebox[\textwidth][c]{
\begin{tabular}{C{0.5}C{0.5}}
\leavevmode%
\beginpgfgraphicnamed{fig-g6-c4i}
\input{tikz/g6-c4i.tikz}
\endpgfgraphicnamed
&
\leavevmode%
\beginpgfgraphicnamed{fig-pappus}
\input{tikz/pappus.tikz}
\endpgfgraphicnamed
\\
(a) & (b)
\end{tabular}
}
\caption{(a) Completing Figure~\ref{fig:g6c4w0}(a)
to obtain the Pappus graph.
(b) The Pappus configuration
with points and lines labelled with the vertex labels of (a).}
\label{fig:pappus}
\end{figure}

\begin{lemma} \label{lem:g6c3}
Let $\G$ be a connected cubic vertex-transitive graph of girth $6$
with signature $(a, b, c)$, where $c = 3$.
Then $a = 2$, $b = 3$
and $\G$ is isomorphic to $\Delta_n$ or $\Sigma_n$ for some $n \ge 4$.
\end{lemma}

\begin{proof}
Lemmas~\ref{lem:eventriineq} and~\ref{lem:azero}
imply that we have $a \in \{1, 2\}$ and $b = a + 1$.
It follows that each vertex is the middle point
of three $2$-paths lying on precisely $a-1$, one and two $6$-cycles.

First, let us prove that no $3$-path lies on two $6$-cycles.
Suppose that the $3$-path $uvwx$ lies on two $6$-cycles.
Then we must have vertices $y, z, y', z'$
such that $uzyxy'z'$ is a $6$-cycle,
see Figure~\ref{fig:g6c3}(a).
Now, all arcs with tail $u$ lie on at least two $6$-cycles,
so we have $a = 2$ and $b = 3$.
In particular, each $2$-path lies on at least one $6$-cycle.
But the $2$-paths $uvw$, $uzy$ and $uz'y'$ lie on two $6$-cycles each,
so all arcs with tail $u$ must lie on three $6$-cycles,
contradiction.

\begin{figure}[t]
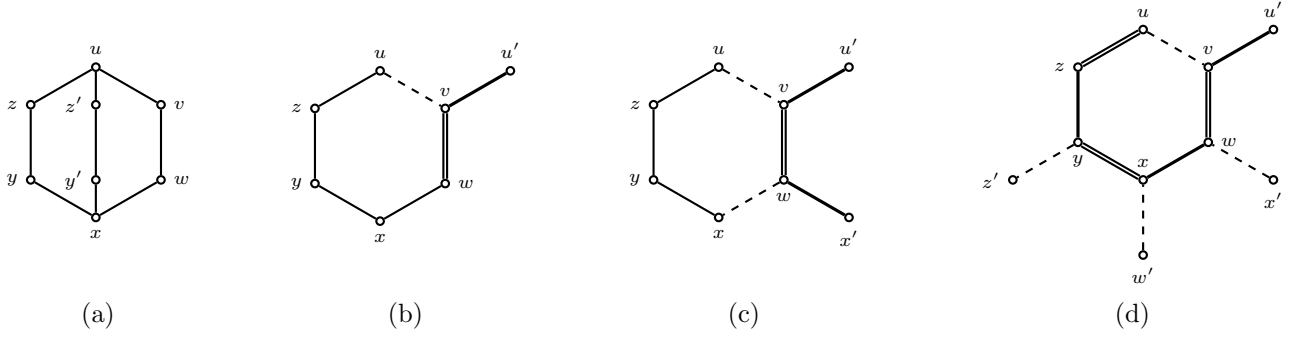

\makebox[\textwidth][c]{
\begin{tabular}{C{0.2}C{0.25}C{0.25}C{0.33}}
\leavevmode%
\beginpgfgraphicnamed{fig-g6-c3a}
\input{tikz/g6-c3a.tikz}
\endpgfgraphicnamed
&
\leavevmode%
\beginpgfgraphicnamed{fig-g6-c3b}
\input{tikz/g6-c3b.tikz}
\endpgfgraphicnamed
&
\leavevmode%
\beginpgfgraphicnamed{fig-g6-c3c}
\input{tikz/g6-c3c.tikz}
\endpgfgraphicnamed
&
\leavevmode%
\beginpgfgraphicnamed{fig-g6-c3d}
\input{tikz/g6-c3d.tikz}
\endpgfgraphicnamed
\\
(a) & (b) & (c) & (d)
\end{tabular}
}
\caption{Constructing a graph of girth $6$ with $c = 3$.
(a) The case when a $3$-path lies on two $6$-cycles,
which leads to a contradiction.
(b) Determining the number of $6$-cycles
the arcs with tail $v$ lie on.
(c) Assuming that both $uv$ and $wx$ lie on $a$ $6$-cycles,
which leads to a contradiction.
(d) Assuming $a = 1$ and $b = 2$, again leading to a contradiction.
The dashed, thick and double edges lie on $a$, $a+1$ and three $6$-cycles,
respectively.}
\label{fig:g6c3}
\end{figure}

Let $H = uvwxyz$ be a $6$-cycle in $\G$.
Suppose that the arc $uv$ lies on precisely $a$ $6$-cycles.
The $2$-path $uvw$ must then lie on precisely one $6$-cycle,
and there is a neighbour $u'$ of $v$
such that the $2$-path $uvu'$ lies on precisely $a-1$ $6$-cycles,
see Figure~\ref{fig:g6c3}(b).
The $2$-path $u'vw$ then lies on two $6$-cycles.

We will now show that the arc $wx$ lies on $a+1$ $6$-cycles.
Suppose that this is not the case.
The arc $wx$ thus lies on $a$ $6$-cycles.
Similarly as before,
the $2$-path $vwx$ also lies on precisely one $6$-cycle
and there is a neighbour $x'$ of $w$
such that the $2$-paths $xwx'$ and $vwx'$
lie on precisely $a-1$ and two $6$-cycles, respectively,
see Figure~\ref{fig:g6c3}(c).
As the $2$-paths $uvw$ and $vwx$ both lie on $H$
and neither of $u'vw$ and $vwx'$ lies on $H$,
it follows that the $3$-path $u'vwx'$ must lie on two $6$-cycles,
contradiction.

Therefore, the edge $wx$ lies on $a+1$ $6$-cycles.
Assume that $a = 1$.
By a similar argument as before,
the arcs $zu$, $yz$ and $xy$
must then lie on three, two and three $6$-cycles,
respectively.
Let $x'$, $w'$ and $z'$ be the remaining neighbours of $w$, $x$ and $y$,
respectively, see Figure~\ref{fig:g6c3}(d).
The $2$-path $x'wx$ cannot lie on any $6$-cycle,
and the $2$-paths $w'xy$ and $xyz'$ must lie on one $6$-cycle each.
By the previous argument,
the latter two $2$-paths must lie on distinct $6$-cycles.
Therefore, the $6$-cycle $H'$ containing the vertices $z'$, $y$, $x$
must also contain the vertex $w$.
However, neither $v$ nor $x'$ can be contained in $H'$ -- contradiction.

It follows that $a = 2$ and $b = 3$,
and each $6$-cycle contains at most $2$ edges
lying on precisely two $6$-cycles.
Let $m$ be the number of edges lying on precisely two $6$-cycles.
Then there are $2m$ edges lying on three $6$-cycles,
and the graph has $2m$ vertices.
As each vertex lies on four $6$-cycles,
the graph $\G$ has precisely $4m/3$ $6$-cycles.
$m$ must then be divisible by $3$
-- in particular, the number of vertices is a multiple of $6$.
Vertex-transitivity implies that
for each two arcs $s, t$ lying on precisely two $6$-cycles,
there is an automorphism $\varphi$ of $\G$ such that $s^\varphi = t$.
Suppose that there is a $6$-cycle
containing a single edge lying on precisely two $6$-cycles.
If each such edge lies on two such $6$-cycles,
then there are $2m > 4m/3$ such $6$-cycles, contradiction.
Therefore, each such edge lies on one $6$-cycle
containing $2$ such edges,
which gives $m$ and $m/2$ $6$-cycles
with $1$ and $2$ such edges, respectively,
again exceeding the total number of $6$-cycles.

It follows that there must be $m$ $6$-cycles
containing $2$ edges lying on two $6$-cycles each,
and $m/3$ $6$-cycles containing no such edges.
Let $W_0 = v_0 w_0 x_0 x_1 w_1 v_1$ be a $6$-cycle
such that the arcs $v_0 v_1$ and $x_0 x_1$ lie on two $6$-cycles each.
Let $u_0$, $y_0$, $u_1$, $w_2$ and $y_1$ be the remaining neighbours
of $v_0$, $x_0$, $v_1$, $w_1$ and $x_1$, respectively.
Since each of the $2$-paths
$w_0 x_0 x_1$, $w_1 x_1 x_0$, $y_0 x_0 x_1$, $y_1 x_1 x_0$
lies on precisely one $6$-cycle,
and the first two lie on $W_0$,
it follows that there must be a $6$-cycle
containing the last two $2$-paths,
say, $Y_0 = x_0 y_0 z_0 z_1 y_1 x_1$.
As the arc $x_0 x_1$ already lies on two $6$-cycles,
$y_1$ cannot be adjacent to $u_0$,
so its remaining neighbour must be a new vertex $y_2$.
By the same argument, $w_2$ is not adjacent to $z_0$
and $y_2$ is not adjacent to $v_0$.
The $6$-cycle containing the $2$-path $x_1 w_1 w_2$
is then $X_1 = w_1 x_1 y_1 y_2 x_2 w_2$,
where $x_2$ is a new vertex.
The second $6$-cycle containing the $2$-path $v_1 w_1 x_1$
must then also contain the vertices $u_1$ and $y_1$,
and by the same argument also the vertex $z_1$.
Therefore, we have $u_1 \sim z_1$,
and by the same argument also $u_0 \sim z_0$.
By similar arguments, we may add new vertices $u_2$, $v_2$ and $z_2$
such that $u_1 \sim u_2$ and $u_2 v_2 w_2 x_2 y_2 z_2$ is a $6$-cycle,
see Figure~\ref{fig:g6c3cyl}(a).

\begin{figure}[t]
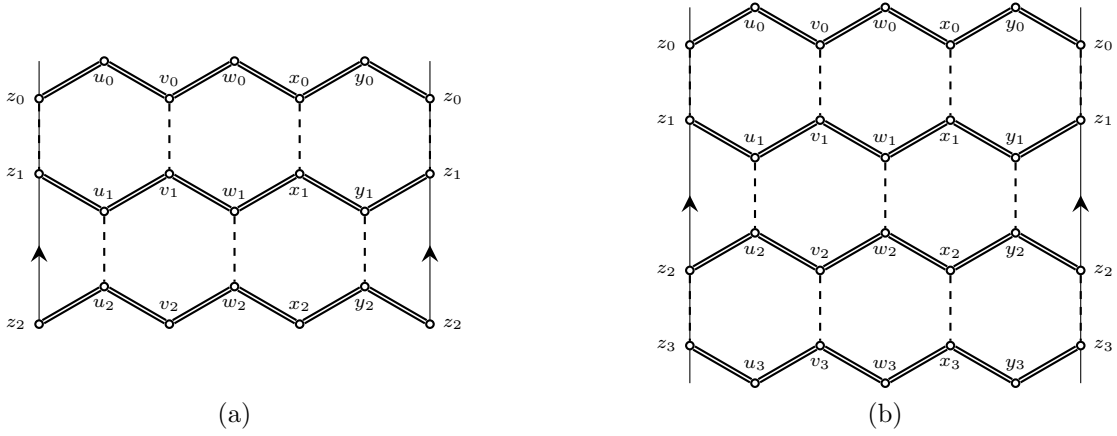

\makebox[\textwidth][c]{
\begin{tabular}{C{0.5}C{0.5}}
\leavevmode%
\beginpgfgraphicnamed{fig-g6-c3e}
\input{tikz/g6-c3e.tikz}
\endpgfgraphicnamed
&
\leavevmode%
\beginpgfgraphicnamed{fig-g6-c3f}
\input{tikz/g6-c3f.tikz}
\endpgfgraphicnamed
\\
(a) & (b)
\end{tabular}
}
\caption{Constructing a graph of girth $6$ with signature $(2, 3, 3)$
embedded on a cylinder.
In (a), the basic structure is shown.
In (b), new vertices have been added
to prevent arcs lying on more than three $6$-cycles.
The dashed and double edges lie on two and three $6$-cycles, respectively.}
\label{fig:g6c3cyl}
\end{figure}

Let $\alpha \in \{u, w, y\}$ and $\beta \in \{v, x, z\}$.
The arcs of form $\alpha_1 \alpha_2$ and $\beta_0 \beta_1$
all already lie on two $6$-cycles,
so we cannot have edges of form $\alpha_0 \beta_2$.
Therefore, we may add a new $6$-cycle $u_3 v_3 w_3 x_3 y_3 z_3$
and arcs of form $\beta_2 \beta_3$,
see Figure~\ref{fig:g6c3cyl}(b).
We have now arrived at a point
where we have determined $12t$ vertices of $\G$ for some $t \ge 2$,
and each of the vertices $\alpha_0$ and $\alpha_{2t-1}$ is missing an arc.

First assume $\alpha_0 \sim \alpha_{2t-1}$ for some $\alpha$
-- without loss of generality, say, $w_0 \sim w_{2t-1}$.
Then each of $u_0$ and $y_0$ must be adjacent
to one of $u_{2t-1}$ and $y_{2t-1}$.
If $u_0 \sim y_{2t-1}$ and $y_0 \sim u_{2t-1}$,
then an automorphism $\varphi$ of $\G$ with $w_0^\varphi = u_0$
has $w_i^\varphi = u_i$ for all $i$ ($0 \le i \le 2t-1$).
We must however also have $w_{2t-1}^\varphi = y_{2t-1}$, contradiction.
Therefore, we have $u_0 \sim u_{2t-1}$ and $y_0 \sim y_{2t-1}$.
Let $n = 2t$.
Identifying the vertices
$u_{2i}$, $v_{2i}$, $w_{2i}$, $x_{2i}$, $y_{2i}$, $z_{2i}$,
$u_{2i+1}$, $v_{2i+1}$, $w_{2i+1}$, $x_{2i+1}$, $y_{2i+1}$ and $z_{2i+1}$
with $\tau_{2i}^+$, $\rho_{2i}^-$, $\tau_{2i}^0$, $\rho_{2i}^+$, $\tau_{2i}^-$, $\rho_{2i}^0$,
$\rho_{2i+1}^+$, $\tau_{2i+1}^-$, $\rho_{2i+1}^0$, $\tau_{2i+1}^+$, $\rho_{2i+1}^-$
and $\tau_{2i+1}^0$ ($0 \le i \le t-1$)
establishes that the graph is isomorphic to $\Sigma_n$,
see for example Figure~\ref{fig:g6c3even}(a).

\begin{figure}[t]
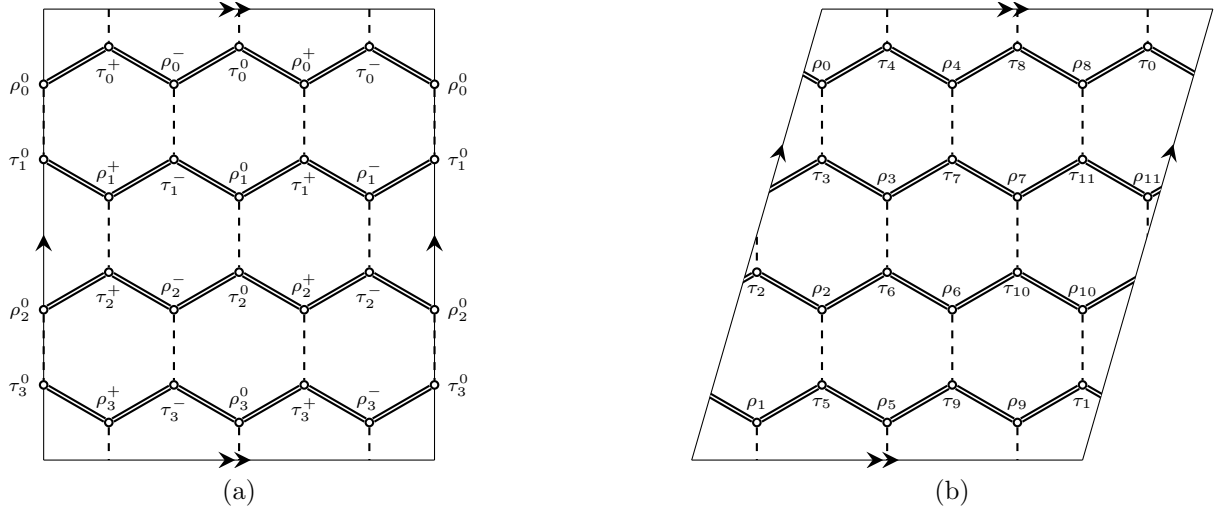

\makebox[\textwidth][c]{
\begin{tabular}{C{0.5}C{0.6}}
\leavevmode%
\beginpgfgraphicnamed{fig-g6-c3g}
\input{tikz/g6-c3g.tikz}
\endpgfgraphicnamed
&
\leavevmode%
\beginpgfgraphicnamed{fig-g6-c3h}
\input{tikz/g6-c3h.tikz}
\endpgfgraphicnamed
\\
(a) & (b)
\end{tabular}
}
\caption{Completing Figure~\ref{fig:g6c3cyl}(b)
to obtain the graphs $\Sigma_4$ in (a) and $\Delta_4$ in (b)
embedded on a torus.
The dashed and double edges lie on two and three $6$-cycles, respectively.}
\label{fig:g6c3even}
\end{figure}

Now assume that $\alpha_0 \sim \alpha'_{2t-1}$
for some distinct $\alpha, \alpha' \in \{u, w, y\}$.
Again, without loss of generality we may assume $u_0 \sim y_{2t-1}$.
Then each of $w_0$ and $y_0$ must be adjacent
to one of $u_{2t-1}$ and $w_{2t-1}$.
By the above argument, $w_0 \sim w_{2t-1}$ is not possible,
so we must have $w_0 \sim u_{2t-1}$ and $y_0 \sim w_{2t-1}$.
Let $n = 2t$ and $k = 3/\gcd(n, 3)$.
Identifying the vertices
$u_{2i}$, $v_{2i}$, $w_{2i}$, $x_{2i}$, $y_{2i}$, $z_{2i}$,
$u_{2i+1}$, $v_{2i+1}$, $w_{2i+1}$, $x_{2i+1}$, $y_{2i+1}$ and $z_{2i+1}$
($0 \le i \le t-1$)
with $\tau_{(2k\mm n)i}$, $\rho_{(2k\mm n)i\mm n}$, $\tau_{(2k\mm n)i\mm n}$,
$\rho_{(2k\mm n)i\pp n}$, $\tau_{(2k\mm n)i\pp n}$, $\rho_{(2k\mm n)i}$,
$\rho_{(2k\mm n)i\pp k\mm n}$, $\tau_{(2k\mm n)i\pp k\mm n}$,
$\rho_{(2k\mm n)i\pp k\pp n}$, $\tau_{(2k\mm n)i\pp k\pp n}$,
$\rho_{(2k\mm n)i\pp k}$ and $\tau_{(2k\mm n)i\pp k\mm n}$
when $n \equiv 2 \pmod{3}$,
and with $\tau_{(2k\mm n)i\pp n}$, $\rho_{(2k\mm n)i\pp n}$,
$\tau_{(2k\mm n)i\mm n}$, $\rho_{(2k\mm n)i\mm n}$, $\tau_{(2k\mm n)i}$,
$\rho_{(2k\mm n)i}$, $\rho_{(2k\mm n)i\pp k}$, $\tau_{(2k\mm n)i\pp k\pp n}$,
$\rho_{(2k\mm n)i\pp k\pp n}$, $\tau_{(2k\mm n)i\pp k\mm n}$,
$\rho_{(2k\mm n)i\pp k\mm n}$ and $\tau_{(2k\mm n)i\pp k}$ otherwise,
establishes that the graph is isomorphic to $\Delta_n$,
see for example Figure~\ref{fig:g6c3even}(b).

If, on the other hand, none of $\alpha_{2t-1}$ are adjacent to known vertices,
we must have a new $6$-cycle $u_{2t} v_{2t} w_{2t} x_{2t} y_{2t} z_{2t}$
and $\alpha_{2t-1} \sim \alpha_{2t}$ for all $\alpha \in \{u, w, y\}$.
If none of $\alpha_0$ is adjacent to known vertices either,
then we must add another $6$-cycle in a similar fashion,
which again gives us the previous case.
Otherwise, first assume, say, $w_0 \sim z_{2t}$.
Then each of $u_0$ and $y_0$ must be adjacent to one of $v_{2t}$ and $x_{2t}$.
If $u_0 \sim v_{2t}$ and $y_0 \sim x_{2t}$,
then an automorphism $\varphi$ of $\G$ with $w_0^\varphi = u_0$
has $w_i^\varphi = u_i$ and then also $z_i^\varphi = x_i$
for all $i$ ($0 \le i \le 2t$).
We must however also have $z_{2t}^\varphi = v_{2t}$, contradiction.
Therefore, we have $u_0 \sim x_{2t}$ and $y_0 \sim v_{2t}$.
Let $n = 2t+1$.
Identifying the vertices
$u_{2i}$, $v_{2i}$, $w_{2i}$, $x_{2i}$, $y_{2i}$, $z_{2i}$
with $\tau_{2i}^+$, $\rho_{2i}^-$, $\tau_{2i}^0$, $\rho_{2i}^+$,
$\tau_{2i}^-$, $\rho_{2i}^0$ ($0 \le i \le t$),
and $u_{2i+1}$, $v_{2i+1}$, $w_{2i+1}$, $x_{2i+1}$, $y_{2i+1}$, $z_{2i+1}$
with $\rho_{2i+1}^+$, $\tau_{2i+1}^-$, $\rho_{2i+1}^0$, $\tau_{2i+1}^+$,
$\rho_{2i+1}^-$, $\tau_{2i+1}^0$ ($0 \le i \le t-1$)
establishes that the graph is isomorphic to $\Sigma_n$,
see for example Figure~\ref{fig:g6c3odd}(a).

\begin{figure}[t]
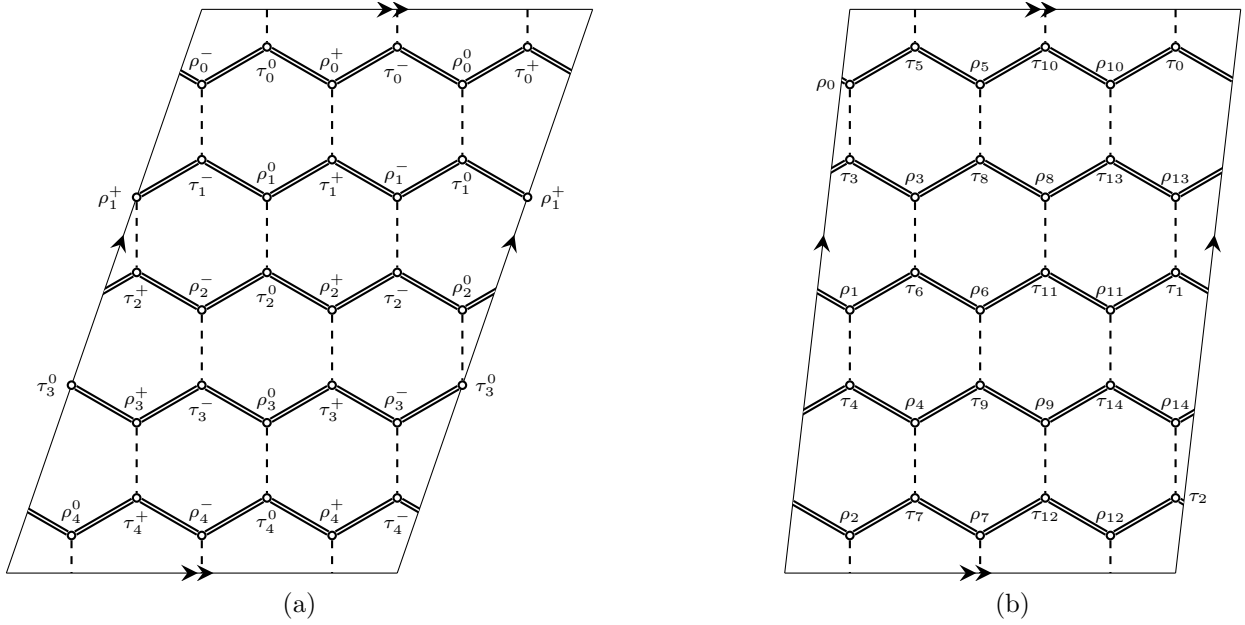

\makebox[\textwidth][c]{
\begin{tabular}{C{0.6}C{0.5}}
\leavevmode%
\beginpgfgraphicnamed{fig-g6-c3i}
\input{tikz/g6-c3i.tikz}
\endpgfgraphicnamed
&
\leavevmode%
\beginpgfgraphicnamed{fig-g6-c3j}
\input{tikz/g6-c3j.tikz}
\endpgfgraphicnamed
\\
(a) & (b)
\end{tabular}
}
\caption{Completing Figure~\ref{fig:g6c3cyl}(b)
to obtain the graphs $\Sigma_5$ in (a) and $\Delta_5$ in (b)
embedded on a torus.
The dashed and double edges lie on two and three $6$-cycles, respectively.}
\label{fig:g6c3odd}
\end{figure}

Finally, assume, say, $y_0 \sim x_{2t}$.
Then each of $u_0$ and $w_0$ must be adjacent
to one of $v_{2t}$ and $z_{2t}$.
By the above argument, $u_0 \sim v_{2t}$ is not possible,
so we must have $u_0 \sim z_{2t}$ and $w_0 \sim v_{2t}$.
Let $n = 2t+1$ and $k = 3/\gcd(n, 3)$.
Identifying the vertices
$u_{2i}$, $v_{2i}$, $w_{2i}$, $x_{2i}$, $y_{2i}$, $z_{2i}$ ($0 \le i \le t$)
and $u_{2i+1}$, $v_{2i+1}$, $w_{2i+1}$, $x_{2i+1}$, $y_{2i+1}$, $z_{2i+1}$
($0 \le i \le t-1$)
with $\tau_{(2k\mm n)i\pp n}$, $\rho_{(2k\mm n)i\pp n}$,
$\tau_{(2k\mm n)i\mm n}$, $\rho_{(2k\mm n)i\mm n}$, $\tau_{(2k\mm n)i}$,
$\rho_{(2k\mm n)i}$ and $\rho_{(2k\mm n)i\pp k}$,
$\tau_{(2k\mm n)i\pp k\pp n}$, $\rho_{(2k\mm n)i\pp k\pp n}$,
$\tau_{(2k\mm n)i\pp k\mm n}$, $\rho_{(2k\mm n)i\pp k\mm n}$,
$\tau_{(2k\mm n)i\pp k}$ when $n \equiv 2 \pmod{3}$,
and with $\tau_{(2k\mm n)i}$, $\rho_{(2k\mm n)i\mm n}$,
$\tau_{(2k\mm n)i\mm n}$, $\rho_{(2k\mm n)i\pp n}$, $\tau_{(2k\mm n)i\pp n}$,
$\rho_{(2k\mm n)i}$ and $\rho_{(2k\mm n)i\pp k\mm n}$,
$\tau_{(2k\mm n)i\pp k\mm n}$, $\rho_{(2k\mm n)i\pp k\pp n}$,
$\tau_{(2k\mm n)i\pp k\pp n}$, $\rho_{(2k\mm n)i\pp k}$,
$\tau_{(2k\mm n)i\pp k\mm n}$ otherwise,
establishes that the graph is isomorphic to $\Delta_n$,
see for example Figure~\ref{fig:g6c3odd}(b).
\end{proof}

We can now wrap up our proof of Theorem~\ref{thm:g6}.
Let $\Gamma$ be a simple connected cubic vertex-transitive graph
and let $(a, b, c)$ be its signature.
Theorem~\ref{thm:cmax} and Lemmas~\ref{lem:g6c7}--\ref{lem:g6c3}
cover the cases when $c \ge 3$.
Of the graphs appearing in these lemmas,
only the Desargues graph does not admit an embedding onto a torus
as a vertex-transitive map of type $\{6, 3\}$.

We are left with the cases when $c \le 2$.
If $(a, b, c) = (2, 2, 2)$,
then, by Theorem~\ref{thm:abc2},
$\G$ is the skeleton of a vertex-transitive map of type $\{6, 3\}$
embedded on a surface of Euler characteristic $\chi = 0$,
i.e., a torus or a Klein bottle.
By Wilson~\cite{w06},
the skeleton of a vertex-transitive map of type $\{6, 3\}$ on a Klein bottle
has girth at most $4$,
so $\G$ must be the skeleton
of a vertex-transitive map of type $\{6, 3\}$ on a torus.
If $(a, b, c) = (1, 1, 2)$,
then, by Theorem~\ref{thm:trieq1},
$\G$ is the skeleton of the truncation
of a connected map of type $\{3, \ell\}$ for some $\ell > 6$.
If $(a, b, c) = (0, 1, 1)$,
then, by Theorem~\ref{thm:azero},
$\G$ is the truncation of a $6$-regular graph $\hat{\G}$
with respect to an arc-transitive dihedral scheme.
This finishes the proof of Theorem~\ref{thm:g6}.

%Note that every map of type $\{6, 3\}$ on a torus is vertex-transitive,
%so its skeleton is a graph from Theorem~\ref{thm:g6}(a)
%if its girth is $6$.
%These graphs are also Cayley graphs.

\section{Larger girths}
\label{sec:g7}

We wrap up this paper with a short discussion
on the problem of extending the results proved here
to graphs of larger girth.
It is not surprising that the complexity of the situation
grows with the girth and that several new infinite families arise,
especially those with a small number of girth cycles,
that is, those with signatures $(a,b,c)$ where $c$ is relatively small.
On the other hand, as computational evidence presented below suggests,
further classification results could be obtained
when one restricts to specific signatures
with large values of $c$ and/or $a,b$.
We leave an in-depth analysis of these cases for future work
and instead list the signatures of graphs of girths $7$, $8$ and $9$
appearing in the census of connected cubic vertex-transitive graphs
on at most $1280$ vertices by Poto\v{c}nik, Spiga and Verret~\cite{psv13}.

Tables~\ref{tab:g7},~\ref{tab:g8} and~\ref{tab:g9}
show the number of connected cubic vertex-transitive graphs
with at most $1280$ vertices for each signature $(a, b, c)$ that appears,
and also the number of symmetric graphs among those
-- clearly, the latter all have $a = b = c$,
so a dash is shown in the other rows in the tables.
Note that substantially more signatures appear for girth $8$
than for girths $7$ and $9$
-- this is mainly due to part \eqref{lem:eventriineq:3}
in Lemma~\ref{lem:eventriineq},
which forbids signatures with $c = a + b$ in graphs of odd girths.
By Theorem~\ref{thm:abc2},
the graphs with signatures $(2, 2, 2)$ are skeletons of maps
-- unlike with girths at most $6$,
there are cases of such maps on nonorientable surfaces.

\begin{table}[h]
\makebox[\textwidth][c]{
\begin{tabular}{|c|c|c|c|}
\hline
signature & vertex-transitive & symmetric & comments \\
\hline
$(0, 1, 1)$ & $76$ & $-$ &
truncations of $7$-regular graphs (Theorem~\ref{thm:azero}) \\
$(2, 2, 2)$ & $8$ & $8$ &
skeletons of maps of type $\{7, 3\}$ (Theorem~\ref{thm:abc2}) \\
$(4, 4, 4)$ & $1$ & $1$ & Coxeter graph \\
$(4, 4, 6)$ & $104$ & $-$ & \\
$(4, 5, 5)$ & $3$ & $-$ & \\
\hline
\end{tabular}
}
\caption{Signatures of cubic vertex-transitive graphs of girth $7$.}
\label{tab:g7}
\end{table}

\begin{table}[h]
\makebox[\textwidth][c]{
\begin{tabular}{|c|c|c|c|}
\hline
signature & vertex-transitive & symmetric & comments \\
\hline
$(0, 1, 1)$ & $7262$ & $-$ &
truncations of $8$-regular graphs (Theorem~\ref{thm:azero}) \\
$(1, 1, 2)$ & $3107$ & $-$ &
truncations of maps of type $\{4, \ell\}$ (Theorem~\ref{thm:trieq1}) \\
$(1, 2, 3)$ & $153$ & $-$ & \\
$(2, 2, 2)$ & $457$ & $21$ &
skeletons of maps of type $\{8, 3\}$ (Theorem~\ref{thm:abc2}) \\
$(2, 2, 4)$ & $1083$ & $-$ & \\
$(2, 3, 3)$ & $1033$ & $-$ & \\
$(3, 3, 4)$ & $51$ & $-$ & \\
$(3, 4, 5)$ & $1$ & $-$ & $\Cay(\D_5 \times \Sym_3,
\{(\tau_0, ()), (\tau_1, (1 \ 2)), (\tau_2, (1 \ 3))\})$ \\
$(4, 4, 4)$ & $108$ & $4$ & \\
$(4, 6, 6)$ & $62$ & $-$ & \\
$(5, 5, 6)$ & $207$ & $-$ & \\
$(6, 6, 6)$ & $1$ & $0$ & $\Cay(\GL(2, 3),
\{{0 \ 1 \choose 1 \ 0}, {1 \ 2 \choose 2 \ 0}, {0 \ 2 \choose 2 \ 2}\})$ \\
$(8, 8, 8)$ & $3$ & $2$ & \\
$(16, 16, 16)$ & $1$ & $1$ & Tutte-Coxeter graph (Theorem~\ref{thm:cmax}) \\
\hline
\end{tabular}
}
\caption{Signatures of cubic vertex-transitive graphs of girth $8$.}
\label{tab:g8}
\end{table}

\begin{table}[h]
\makebox[\textwidth][c]{
\begin{tabular}{|c|c|c|c|}
\hline
signature & vertex-transitive & symmetric & comments \\
\hline
$(0, 1, 1)$ & $51$ & $-$ &
truncations of $9$-regular graphs (Theorem~\ref{thm:azero}) \\
$(2, 2, 2)$ & $156$ & $12$ &
skeletons of maps of type $\{9, 3\}$ (Theorem~\ref{thm:abc2}) \\
$(2, 3, 3)$ & $3$ & $-$ & \\
$(4, 4, 4)$ & $2$ & $2$ & \\
$(6, 6, 6)$ & $5$ & $3$ & \\
$(8, 8, 8)$ & $2$ & $1$ & Biggs-Smith graph (symmetric case) \\
\hline
\end{tabular}
}
\caption{Signatures of cubic vertex-transitive graphs of girth $9$.}
\label{tab:g9}
\end{table}

%\bibliographystyle{abbrv}
%\bibliography{reference}

\end{document}